\documentclass{amsart}
\usepackage{tikz}
\usepackage{xcolor}
\usepackage{amssymb,latexsym,amsmath,extarrows}
\usepackage{graphicx,mathrsfs}
\usepackage{hyperref,url}
\numberwithin{equation}{section}

\usepackage{enumerate}

\newtheorem{theorem}{Theorem}[section]
\newtheorem{lemma}[theorem]{Lemma}

\newtheorem{proposition}[theorem]{Proposition}
\newtheorem{remark}[theorem]{Remark}
\newtheorem{question}[theorem]{Question}

\newtheorem{definition}[theorem]{Definition}

\newtheorem{conjecture}[theorem]{Conjecture}

\newcommand{\al}{\alpha}
\newcommand{\be}{\beta}

\newcommand{\e}{\varepsilon}

\newcommand{\ka}{\kappa}
\newcommand{\la}{\lambda}

\newcommand{\si}{\sigma}
\newcommand{\Si}{\Sigma}
\newcommand{\vp}{\varphi}
\newcommand{\om}{\omega}

\newcommand{\cb}{\mathcal B}

\newcommand{\cj}{\mathcal J}

\newcommand{\wt}{\widetilde}
\newcommand{\wh}{\widehat}

\newcommand{\ZR}{\mathbb{R}}
\newcommand{\ZT}{\mathbb{T}}
\newcommand{\ZZ}{\mathbb{Z}}

\newcommand{\ZN}{\mathbb{N}}
\newcommand{\ZS}{\mathbb{S}}
\newcommand{\ti}{\tilde}
\newcommand{\Id}{{\bf 1}}

\newcommand{\cT}{{\mathcal T}}
\newcommand{\cQ}{{\mathcal Q}}
\newcommand{\cl}{{\mathcal L}}

\newcommand{\dist}{{\rm{dist}}}

\newcommand{\Br}{{\rm{br}}}
\newcommand{\supp}{{\rm supp}}

\makeatletter
\def\l@subsection{\@tocline{2}{0pt}{2.5pc}{3pc}{}}
\makeatother

\title{On almost everywhere convergence of planar Bochner-Riesz means}

\author{Xiaochun Li} \address{Xiaochun Li\\  Department of Mathematics\\ University of Illinois Urbana-Champaign, USA} \email{xcli@illinois.edu}

\author{Shukun Wu} \address{ Shukun Wu\\  Department of Mathematics\\ Indiana University Bloomington, USA} \email{shukwu@iu.edu}

\begin{document}

\begin{abstract}
We demonstrate the almost everywhere convergence of the planar Bochner-Riesz means for $L^p$ functions in the optimal range when $5/3\leq p\leq 2$.  
This is achieved by establishing a sharp $L^{5/3}$ estimate for a maximal operator closely associated with the Bochner-Riesz multiplier operator. The estimate depends on a new refined $L^2$ estimate, which may be of independent interest.

\end{abstract}

{
\stepcounter{footnote}
}

\maketitle

\tableofcontents

\section{Introduction}\label{Intro}

\subsection{Background}

For any Schwartz function $f$, the $n$-dimensional Bochner-Riesz means $T^\la_tf$ is defined as
\begin{equation}
\label{br-mean}
    T_t^\la f(x):=(2\pi)^{-n}\int_{\mathbb{R}^n} \Big(1-\frac{|\xi|^2}{t^2}\Big)^\la_+\wh{f}(\xi)e^{ix\cdot\xi}d\xi.
\end{equation}
It can be extended to a multiplier operator with the Fourier multiplier $(1-|\cdot|^2/t^2)^\la_+$. 
Analogous to the Gauss means (where the multiplier is replaced by $e^{-|\cdot|^2/t}$) or the Abel means (where the multiplier is replaced by $e^{-|\cdot|/t}$), the Bochner-Riesz means were introduced as a summation method arising from the study of radial convergence of the Fourier transform. 
The issue of almost everywhere convergence concerning the Bochner-Riesz means stands out as one of the most intriguing and significant problems in modern analysis. It can be precisely formulated as follows:

\begin{question}
For any $L^p$ function $f$, what is the optimal regularity $\la$ required for the Bochner-Riesz means $T_t^\la f$ to converge to $f$ almost everywhere?
\end{question}

In the higher range $p\geq2$, the almost-everywhere convergence problem is completely solved in \cite{CRV}. However, the lower range $p<2$ poses significantly greater challenges. 
It is worth noting that, when $p<2$, Stein's maximal principle suggests that the following two claims are equivalent (see \cite{Tao-maximal}):
\begin{enumerate}[ i)]
    \item $\lim_{t\to\infty }T_t^\la f(x) 
    {=}f(x)$ for a.e. $x\in \mathbb R^n$ and  every $L^p$ function $f$.
    \item The maximal Bochner-Riesz operator $T^\la_\ast$ is bounded in $L^p$.
\end{enumerate}
Here the maximal operator $T^\la_\ast$ is defined as
\begin{equation}
\label{MBR}
   T^{ \la}_\ast f(x):=\sup_{t>0}|T_t^\la f(x)|.
\end{equation}

Regarding $L^p$ estimates for the maximal operator, Tao \cite{Tao-weak-type-BR} proposed the following conjecture.

\begin{conjecture}\label{1.2}
\label{mbr-conj}
For any $1\leq p<2$, the maximal Bochner-Riesz operator $T^\la_\ast$ extends to a bounded operator in $L^p(\ZR^n)$ when $\la>\max\{0, \frac{2n-1}{2p}-\frac{n}{2}\}$, that is, 
\begin{equation}
\label{mbr-conj-esti}
    \big\| T^\la_\ast f\big\|_{L^p(\mathbb R^n)}\leq C \|f\|_{L^p(\mathbb R^n)}\,.
\end{equation}
\end{conjecture}

The conjecture has remained remarkably intriguing and challenging, with very limited progress made over the years.
Its motivation is rooted in the pointwise convergence of circular Fourier partial sums. Fefferman's renowned example \cite{Fefferman-Ball-multiplier}
shows that Conjecture \ref{1.2} cannot hold when $\lambda=0$. 
Consequently, when $p \neq 2$, the pointwise convergence problem for the Bochner-Riesz means is meaningful only for $\lambda \neq 0$. The most challenging scenario arises when $\lambda$ is close to zero and $p = 2 - \frac{1}{n}$. 
Prior to our work, there had been no positive conclusion for $\la$ near zero. Additionally, for higher dimensional cases when $n\geq 4$, no affirmative (non-trivial) result has been reached even when $\la$ is away from zero.
In the planar case, we are positioned to offer a partial resolution of this conjecture.

\begin{theorem}
\label{main}
Conjecture \ref{mbr-conj} is true when $n=2$ and $5/3\leq p\leq 2$.
\end{theorem}

\begin{figure}
\begin{tikzpicture}
\draw[thick] (0,0) -- (9,0) node [anchor=north west] {};
\draw[thick] (9,0) -- (9,9) node [anchor=north east] {};
\node at (4.5,-1) {\footnotesize $1/p$};
\node at (0,-0.3) {\small $\frac{1}{2}$};
\node at (3,-0.3) {\small $\frac{2}{3}$};
\node at (9,-0.3) {\small $1$};
\node at (9.3,9) {\small $\frac{1}{2}$};
\node at (10,9) {\footnotesize $\la$};
\draw (3,0) -- (9,9);
\draw (0,0) -- (9,9);

\draw[dotted] (0,0) -- (2.93,1.88);
\draw[dotted] (4,3) -- (2.93,1.88);
\draw[dotted] (4,3) -- (9,9);
\draw [fill=gray, opacity=0.4] (9,9) to (1.9,0) to  (0,0) to (2.93,1.88) to (4,3) to (9,9);

\node at (7,3) {A};
\node at (3.4,1.3) {B};
\node at (3.5,3) {C};

\node at (1.9,-0.3) {\small $\frac{3}{5}$};
\draw (1.9,0) -- (1.9,0.1);
\draw[fill] (3.6,2.7) circle (.25mm);
\draw[fill] (3.6,2.6) circle (.25mm);
\draw[fill] (3.6,2.7) circle (.25mm);
\draw[fill] (4,3) circle (.25mm);
\draw[fill] (2.93,1.88) circle (.25mm);
\draw (1.9,0) -- (9,9);
\draw [fill=gray, opacity=0.2] (9,9) to (0,0) to (1.9,0) to  (9,9) ;
\end{tikzpicture}
\caption{$L^p$ behavior of the maximal Bochner-Riesz operator}
\label{fig1}
\end{figure}

Theorem \ref{main} is the first instance of attaining sharp and non-trivial conclusions for Conjecture \ref{mbr-conj} when $p<2$ and $\lambda$ near zero. As a direct implication,  within the range $5/3\leq p\leq2$, the Bochner-Riesz means $T^\la_t f$ converge to $f$ for any $f\in L^p(\ZR^2)$ if $\la>0$. 
 Theorem \ref{main} encompasses all prior findings in the plane in \cite{Tao-MBR,Li-Wu-MBR,GW,Kim}, which collectively affirm the pointwise convergence conclusion for 
$(1/p, \lambda)$ within the quadrilateral region $C$ depicted in Figure \ref{fig1}.
Theorem \ref{main} extends the known region to a significantly larger area,  specifically, the triangular region defined by three points: $(1/2,0), (3/5, 0)$, and $ (1, 1/2)$. 
Region $B$ remains open, while Region $A$ is invalidated by Tao's example, where, after suitable reductions on $T^\la_\ast$, one tests \eqref{max-p-Sf} with $f(x_1,x_2)=a(R^{-1/2}x_1)a(x_2)e^{ix_1}$ for a standard bump function $a$. 

\medskip

Our approach diverges from prior studies.
The key ingredient is a quantitative version of weighted $L^2$ estimates, where the weight is an arbitrary union of unit balls. 
These refined $L^2$ estimates extend beyond addressing the maximal Bochner-Riesz operators, such as the reverse square function inequality in the plane, suggesting broader applications for the quantitative weighted $L^2$ inequalities. We anticipate further applications stemming from these insights.

\medskip

Our main theorem, Theorem \ref{main}, is established through $L^p$ estimates of a specific maximal operator, which is closely related to the classical Bochner-Riesz multiplier. Before proceeding, let's introduce the definition of this maximal operator.
For $t\in [1,2]$ and any positive real number $R$, define $S_{t, R}$ by
\begin{equation}\label{defS-tR}
\wh{S_{t, R}f}(\xi) = a(R(t-|\xi|)) \wh f(\xi)\,
\end{equation}
for any $\xi\in\mathbb R^2$ and any Schwartz function $f$, where the function $a$ represents a standard bump function on $[-1,1]$.  
The multiplier $a(R(t-|\xi|))$ here is a bump function supported on the annulus 
$A_{t, R}:=\{\xi\in\mathbb R^2: t-R^{-1}\leq |\xi| \leq t+R^{-1}\}$.  We define the corresponding maximal operator
$S^*_R$ by,  
\begin{equation}\label{Def-SR}
 S^*_R f(x) =\sup_{t\in [1,2]} \big| S_{t, R}f(x)\big|\,,
\end{equation}
for any $x\in\mathbb R^2$.

\smallskip

By means of particular simplifications detailed in Section \ref{linearization-section}, establishing Theorem \ref{main} can be accomplished by focusing solely on the following $L^p$-estimate of the maximal operator, which formulates our central theorem in this article. 

\begin{theorem}\label{Lp-max}
For any $\e>0$ and any $5/3\leq p\leq 2$,  there is a constant $C_\e$ such that 
\begin{equation}\label{max-p-Sf}
 \big\| S^*_R f\big\|_{L^p(\mathbb R^2)}\leq C_\e R^\e \|f\|_{L^p(\mathbb R^2)}\,,
\end{equation}
for any Schwartz function $f$ and any $R\geq 1$. 
\end{theorem}

In the planar case, Tao’s conjecture would be settled if one proved the following conjecture.
\begin{conjecture}
\label{plane-conj}
The $L^p$-estimate (\ref{max-p-Sf}) holds for $p=3/2$.  
\end{conjecture}

\medskip 

One of the main challenges posed by Conjecture \ref{mbr-conj} arises from the absence of orthogonality in the radial direction. Research by \cite{Kim-Seeger} demonstrates that substituting the $L^p_xL^\infty_t$ norm in \eqref{mbr-conj-esti} with a larger $L^p_xL^{p'}_t$ norm ($1/p+1/p'=1$) leads to the breakdown of Conjecture \ref{mbr-conj} within the entire isosceles right triangular area illustrated in Figure \ref{fig1}. This essentially means that an interpolation between the established outcomes at the endpoints $L^1_xL^\infty_t$ and $L^2_xL^2_t$ offers an almost sharp bound for the $L^p_xL^{p'}_t$ norm. To extend the results for the maximal problem beyond this interpolation line, it becomes imperative to establish a certain level of radial orthogonality. As we'll explain shortly, this is largely achieved implicitly through the consideration of specific weighted estimates.


\begin{remark}
\rm
Surprisingly, Tao's example $f(x_1,x_2)=a(R^{-1/2}x_1)a(x_2)e^{ix_1}$, after some reductions, also serves as a sharp example of this $L^p_xL^{p'}_t$ problem. In Section \ref{ending-remark}, we will briefly explain how our approach in this paper can, to some extent, rule out Tao's example for Conjecture \ref{mbr-conj}.

\end{remark}

\subsection{A Weighted Estimation Approach}
By employing certain reductions (refer to Section \ref{linearization-section}), demonstrating Theorem \ref{Lp-max} can be narrowed down to primarily examining 
the $L^p$ estimates for the following model operator:
\begin{equation}
\label{model-operator-intro}
    \sum_{j\in[1,R]\cap\ZZ}S_{j} f\Id_{F_j}\,,
\end{equation}
where $F_j$'s are disjoint sets in an $R$-ball $B_R$ in the plane, each of them is a collection of unit balls, $\Id_{F_j}$ represents the indicator function of $F_j$, and $S_j$ is a multiplier operator given by $\wh{S_jf}(\xi)=a(R(t_j-|\xi|))\wh f(\xi)$ for some $t_j\in[1,2]$. 
Notice that the operator $S_j$ depends on both $j$ and $R$, yet we abstain from explicitly including $R$ in the notation $S_j$.\\

More precisely, it suffices to establish that 
\begin{equation}
    \big\|\sum_{j}S_{j} f\Id_{F_j}\big\|_p \leq C_\e R^\e\|f\|_p,
\end{equation}
for one endpoint with $p=5/3$ and any $\e>0$, as the $L^p$ estimates (\ref{max-p-Sf}) can then be derived through interpolation with the other known endpoint at $p=2$.
This observation leads us to explore the weighted $L^p$-norm $\|S_{j}f\|_{L^p(\Id_{F_j})}$ for any fixed $j$. Given that each $S_j$ behaves similarly to the spherical multiplier operator $S:=S_{1, R}$, we are interested in analyzing the $L^p$-norm restricted to a subset $F\subset B_R$, that is,   
\begin{equation}
\label{weighted-lp}
    \|Sf\Id_{F}\|_p.
\end{equation}

\smallskip

Decoupling inequalities have been utilized as a method in previous studies aiming towards \eqref{weighted-lp}. In \cite{Tao-MBR,Li-Wu-MBR}, \eqref{weighted-lp} is bounded by a decoupling norm, albeit without incorporating the information from the weight $\Id_{F}$. In a more recent paper \cite{GW}, the authors proposed to bound this weighted $L^p$-norm using a decoupling norm with a saving term expressed as
\begin{equation}\label{g}
\Big(\frac{|F|}{R^n}\Big)^\al
\end{equation}
for some $\al>0$. Such decoupling inequalities have contributed to advancements regarding Conjecture \ref{mbr-conj} in dimensions two and three. However, for the maximal Bochner-Riesz problem, the decoupling norm might not be the most suitable choice. One reason is that decoupling inequalities are powerful tools primarily aimed at analyzing constructive interference, which is often the main challenge in $L^p$-problems when $p>2$. However, in Conjecture \ref{mbr-conj}, we are required to consider $L^p$ functions in the lower range of $p$, where the constructive interference might not be the main obstacle. 

\medskip

Motivated by the work \cite{GW},  which established a gain of the form \eqref{g} in $L^p$ spaces, we aim to obtain a similar gain for the weighted estimate of $Sf$ in the $L^2$ space, the most natural space. 
Unfortunately, the Knapp example $f(x)=e^{itx_n}a(R^{-1/2}\bar x)a(R^{-1}x_n)$ with $(\bar x, x_n)\in\mathbb R^{n-1}\times \mathbb R$ demonstrates that we cannot generally expect a weighted inequality in the form
\begin{equation}
    \|Sf\Id_{F}\|_{L^2(\mathbb R^n)}\lesssim \Big(\frac{|F|}{R^n}\Big)^\al \|Sf\|_2.
\end{equation}
The first observation we made is that if the operator $S$ is replaced with its multilinear analog $``\text{Mul} (S)"$ (heuristically, $\text{Mul}(S) f$ can be viewed as the geometric average of $S_k f$ for $k = 1, \ldots, n$, where the Fourier transforms of the ${S_k f}$ are supported in thin neighborhoods of $\mathbb{S}^{n-1}$ that are quantitatively transverse    ), then the Bennett-Carbery-Tao multilinear restriction theorem \cite{BCT} indicates that
\begin{equation}
\label{multi-weighted-l2}
    \|\text{Mul}(S)f\Id_{F}\|_{L^2(\mathbb R^n)}\lessapprox\Big(\frac{|F|}{R^n}\Big)^{\frac{1}{2n}}\|Sf\|_2,
\end{equation}
which has the desired form for the estimates we are seeking.
We remark that the multilinear weighted $L^2$ estimate \eqref{multi-weighted-l2} is sharp, by testing $F=B_r$ for any $1\leq r \leq R$. 
Nevertheless, we believe these are essentially the only sharp examples, and our subsequent improvement is guided by this heuristic.

\medskip

The above observation is already quite useful.
In fact, it suggests that if $S$ could be substituted freely with its multilinear analog, then \eqref{multi-weighted-l2} would prove some desirable results for the maximal operator $S^\ast_R$. 
To see why, let us focus on the two-dimensional case, and take $n=2$ in \eqref{multi-weighted-l2}.
We first assume 
\begin{equation}\label{j}
    (\# j)|F_j|\sim R^2,
\end{equation}
which can be achieved by pigeonholing. 
Thus, on the one hand, we have from the weighted $L^2$ estimate (\ref{multi-weighted-l2}) and the assumption (\ref{j}), 
\begin{equation}
    \sum_j\big\|S_jf\Id_{F_j}\big\|_2^2\lesssim\sum_j\!\big\|\text{Mul}(S_j)f\Id_{F_j}\big\|_2^2 \lessapprox\Big(\frac{|F_j|}{R^2}\Big)^{\!\frac{1}{2}}\!\sum_j\!\big\|S_jf\big\|_2^2\lesssim \frac{1}{(\# j)^{\frac{1}{2}}}\big\|f\big\|_2^2.
\end{equation}
On the other hand, invoke the classical $L^{4/3}$ estimate for the Bochner-Riesz operator in \cite{Carleson-Sjolin} to obtain
\begin{align}
\sum_j\big\|S_jf\Id_{F_j}\big\|_{4/3}^{4/3}\lesssim\sum_j\big\|f\big\|_{4/3}^{4/3}\lesssim(\#  j)\big\|f\big\|_{4/3}^{4/3}.
\end{align}
An interpolation between the above two estimates gives
\begin{equation}
\label{j-final}
    \sum_j\big\|S_jf\Id_{F_j}\big\|_{p}^{p}\lessapprox\big\|f\big\|_p^p
\end{equation}
for $p=16/9$, which proves Conjecture \ref{mbr-conj} when $n=2$ and $16/9\leq p\leq 2$.

\medskip

To make use of the above observation, the first obstacle we encounter is how to transition from the linear operator to the bilinear one. 
Typically, the Bourgain-Guth broad-narrow method (see \cite{Bourgain-Guth-Oscillatory}) is a standard tool for this purpose.
However, unlike the Fourier restriction conjecture, Theorem \ref{Lp-max} concerns $L^p$ inequalities for $p<2$. 
Due to the lack of orthogonality when $p<2$, there is inefficiency in the standard broad-narrow argument. 
In particular, the ``narrow part" cannot be easily handled and summed by a rescaling argument.

Before a discussion on how we address this broad-narrow issue, let us introduce the notation of ``broadness".
This is similar to the concept introduced by Guth \cite{Guth} in the Fourier restriction theory.

\begin{definition}\label{def1.7}
We call an arc $\tau\subset\ZS^1$ a {\bf $\rho$-cap} if the Lebesgue measure of the cap $\tau$ is $\sim \rho$. 
We also define  
\begin{equation}\label{Def1.6}
    C_\tau:=\{\xi\in\ZR^2:\xi/|\xi|\in\tau\}
\end{equation}
as the conic region determined by $\tau$.
\end{definition}

For any cap $\tau$,  $S_\tau f$ denotes the smooth Fourier restriction of $Sf$ in $\tau$, that is, $\wh{S_\tau f} = a_\tau \wh {Sf}$. 
Here $a_\tau$ is a smooth bump function on the conic region $C_\tau$. 

\begin{definition}
\label{def-broad-norm}
Suppose $\Si=\{\si\}$ is a collection of finitely overlapping $(\log R)^{-1}$-caps. Let $M$ be a large number. For any function $f$ and any set $E\subset B_R$, define the {\bf broad norm} $\|Sf\|_{L^2_{\Br_M}(E)}$ as 
\begin{equation}
    \|Sf\|_{L^2_{\Br_M}(E)}:=\max_{\substack{\Si'\subset\Si,\\ \#\Si'=M}}\min_{\si\in\Si'}\{\|S_{\si}f\|_{L^2(E)}\}.
\end{equation}
\end{definition}

The quantity $\|Sf\|_{L^2_{\Br_M}(E)}$ is determined by the cap $\si$ such that $\|S_{\si}f\|_{L^2(E)}$ is the $M$-th largest element among the family $\{\|S_{\si}f\|_{L^2(E)}\}_{\si\in \Si}$. 
When $E$ is a unit ball, the definition coincides with the one introduced in  \cite{Guth}. 
Compared to the standard $L^2$ norm $\|Sf\|_{L^2(E)}$ or its bilinear analog $\|\text{Mul}(S)f\|_{L^2(E)}$, the broad norm $\|Sf\|_{L^2_{\Br_M}(E)}$ possesses a stronger directional non-concentration property. 

\medskip

To deal with the inefficiency mentioned above regarding the broard-narrow argument, we repeatedly apply the broad-narrow argument to the model function $\sum _j S_j f\Id_{F_j}$ until it is $\al$-broad (see Lemma \ref{broad-narrow}). 
As mentioned earlier, there is a loss when summing over all the $\alpha$-broad parts, each corresponding to a thin tube of dimensions $\alpha R \times R$ that arises naturally from the pseudo-local property of the multiplier operator localized near an arc of length $\alpha$
To pick up the loss, we use a local $L^2$-estimate (Lemma \ref{local-L2-0}) inside each $\al R$-ball, and use a Kakeya-type inequality (Lemma \ref{kakeya-type-lem}) to control the interaction between the $\al R$-balls and $\al R\times R$-tubes. 

In each of the $\al$-broad parts, we want to use the bilinear weighted $L^2$-estimate \eqref{multi-weighted-l2} and the strategy described from \eqref{j} to \eqref{j-final}. But this only gives us a weak conclusion that cannot even cover the earliest result by Tao \cite{Tao-MBR}. To refine this strategy, we consider the dual function of $\sum _j S_j f\Id_{F_j}$, $\sum _j S_j f_j$ with $\supp{f_j}\subset F_j$, in the dual space $L^4$. Then, among other things, we essentially bound the $L^4$-norm of  $\sum _j S_j f_j$ by the square function $\big(\sum_j\|S_j f_j\|_2^2\big)^{1/2}$. 
Since each $f_j$ still retains the information $\supp(f_j)\subset F_j$, we naturally want to make use of the weighted $L^2$-estimate \eqref{multi-weighted-l2} again. 
Unfortunately, after applying the duality, we lose  the essential bilinear structure of the function $S_j f_j$, which is critical in \eqref{multi-weighted-l2}.
This loss complicates the application of \eqref{multi-weighted-l2}. Without the bilinear structure, our approach can only reproduce the earliest result \cite{Tao-MBR}, and it cannot account for the more recent ones \cite{Li-Wu-MBR, GW, Kim}.  

The above discussion urges us to establish a refinement on \eqref{multi-weighted-l2}, which becomes a cornerstone of our proof of Theorem \ref{Lp-max}.  To elucidate our new refined $L^2$ results, we first need to introduce some definitions. We denote the $\rho$-neighborhood of a set $E$ by $N_\rho(E)$ and the Lebesgue measure of $E$ by $|E|$. Moreover, $A\gg B$ signifies that the number $A$ is much greater than the number $B$, and $\# S$ represents the number of elements in a (finite) set $S$. 

\begin{definition}
\label{kappa-regular}
Let $\kappa\in [r^{-1/2},1]$ for some given real number $r\geq 1$. A subset $E$ of an $r$-ball   $B_r$ in the plane is called {\bf{$\ka$ regular}} in $B_r$  if
\begin{enumerate}
    \item For any $r^{1/2}\times r$-tube $T$ in $\mathbb R^2$,
    \begin{equation}
        |N_{r^{1/2}}(E)\cap T|\lesssim \ka |T|.
    \end{equation}
    \item There exists a collection $\ZT_{\ka} $ of $r^{1/2}\times r$ tubes  with $\#\ZT_{\ka}\lesssim|N_{r^{1/2}}(E)|/(\ka r^{3/2})$ so that $E\subset \bigcup_{T\in\ZT_{\ka}}T$.
\end{enumerate}
\end{definition}

\begin{definition}
Given a large number $R\gg K\gg1$ with $K=R^{2^{-n}}$, let $r_1<\cdots<r_n$ be scales such that $r_n=R$, $r_k=r_{k+1}^{1/2}$ for all $k\in \{1, \cdots, n-1\}$. 
Let $\ka_{r_k}\in [r_k^{-1/2}, 1]$ for $k=1,\ldots n$.
We say a set $E$ is {\bf{regular with factors $(\ka_{r_1},\ldots,\ka_{r_n})$}} if for every $k\in \{1, \cdots, n\}$ the following is true.
\begin{enumerate}
    \item $E$ is $\ka_{r_k}$ regular in an $r_k$-ball   $B_{r_k}$ whenever $B_{r_k}\cap E\not=\varnothing$. 
    \item $E$ admits a maximal covering by disjoint $r_k$-balls $Q_{r_k}$, for every such $r_k$-ball $Q$, the quantity $|N_{r_{k-1}}(E)\cap Q_{r_k}|$ is the same up to a constant multiple which may depend on $k$. 
\end{enumerate}

\end{definition}

Note that $\sum_{Q_{r_k}}|N_{r_{k-1}}(E)\cap Q_{r_k}| =   |N_{r_{k-1}}(E)|$ and $|N_{r_k}(E)|\sim |Q_{r_k}|\#\{Q_{r_k}: Q_{r_k}\cap N_{r_{k-1}}(E)\neq \emptyset\} $, where $\{Q_{r_k}\}$ is a maximal covering of $E$ by $r_k$-balls.
Thus, for each such $Q_k$, we have
\begin{equation}
\label{uniform-assump}
    \frac{|N_{r_{k-1}}(E)\cap Q_{r_k}|}{|Q_{r_k}|}\sim\frac{|N_{r_{k-1}}(E)|}{|N_{r_{k}}(E)|},
\end{equation}

\begin{proposition}
\label{l2-refine-prop}
Let $R\gg K\gg 1$. 
Let $F$ be a union of finite-overlapping $K$-balls $\{B\}$.
Suppose $F$ is regular with factors $(\ka_{r_1},\ldots, \ka_{r_n})$ and $M\geq 3n$. 
Then for any function $f\in L^2$,   we have
\begin{equation}
\label{want-1}
    \sum_{B\subset F}\big\|Sf\big\|_{L^2_{\Br_M}(B)}^2\lesssim(\log R)^{5n}\prod_{k=1}^n\ka_{r_k}^{-1}\Big(\frac{|F|}{R^2}\Big)\big\|Sf\big\|_2^2.
\end{equation}
and
\begin{equation}
\label{want-2}
    \big\|Sf\big\|_{L^2(F)}^2\lesssim\prod_{k=1}^n\ka_{r_k}\big\|Sf\big\|_{2}^2.
\end{equation}
In particular, the geometric mean of (\ref{want-1}) and (\ref{want-2}) yields that 
\begin{equation}
\label{want-3}
    \sum_{B\subset F}\big\|Sf\big\|_{L^2_{\Br_M}(B)}^2\lesssim(\log R)^{5n}\Big(\frac{|F|}{R^2}\Big)^{\frac{1}{2}}\big\|Sf\big\|_2^2.
\end{equation}
\end{proposition}

\begin{remark}

\label{refine-l2-remark}

\rm

Although Proposition \ref{l2-refine-prop} is stated with $Sf$ being the smooth Fourier restriction of $f$ on the $R^{-1}$-neighborhood of the unit circle, it is also valid when the unit circle is replaced by a $C^2$ curve with nonzero curvature.

\end{remark}

Proposition \ref{l2-refine-prop} enables us to derive the following refinement of \eqref{multi-weighted-l2} in $\mathbb{R}^2$: for any $L^2(\mathbb{R}^2)$ functions $f_1, f_2$, we have
\begin{equation}
\Big( \sum_{B\subset F}\big\|Sf_1\big\|_{L^2_{\Br_M}(B)}^2\Big)^{1/2}\|Sf_2\Id_{F}\|_2\lessapprox\Big(\frac{|F|}{R^2}\Big)^{\frac{1}{2}}\|f_1\|_2\|f_2\|_2.
\end{equation}
This improvement over \eqref{multi-weighted-l2} is notable since it eliminates the requirement for a bilinear structure in the second term $\|Sf_2\Id_{F}\|_2$ on the left side. 
Consequently, we can eventually derive an $L^{5/3}$ estimate in Theorem \ref{Lp-max}. 
Let us give a quick explanation of how this is achieved below.

\medskip

\subsection{Sketch of the proof of Theorem \ref{Lp-max}}
The proof is technically involved and intricate, so we begin by outlining the main ideas, with an emphasis on how Proposition \ref{l2-refine-prop} leads to Theorem \ref{Lp-max}.

\smallskip

By a broad-narrow argument, we can identify a factor $\al\in[R^{-1/2},1]$ such that the model operator appearing in \eqref{model-operator-intro} is ``$\al$-broad" (see Lemma \ref{broad-narrow}).
In other words, for all $x\in B_R$, there exists only one arc $\tau$ of length $\al$ so that, up to negligible losses,
\begin{equation}
\label{model-operator-intro-2}
    S^\ast_R f\approx\sum_{j}S_{j} f\Id_{F_j}(x)\approx \sum_{j}S_{\tau,j} f\Id_{F_j}(x).
\end{equation}
Here, recalling Definition \ref{def1.7}, $\wh{S_{\tau,j} f} := a_\tau \wh {S_jf}$, where $a_\tau$ is a smooth bump function on the conic region $C_\tau$.
Let $\ZT_\tau$ be a family of finite-overlapping rectangles of dimensions $\al R\times R$ with direction $\tau$.
Since the multiplier of $S_{\tau,j}$ decays rapidly outside an $\al R\times R$-rectangle with direction $\tau$ centered at the origin, the functions $\{\sum_{j}S_{\tau,j} f\Id_{F_j\cap T}:T\in\ZT_\tau\}$ are essentially independent.
After several steps of uniformization (see Proposition \ref{reduction-lem}), we may assume that for all $T\in\ZT_\tau$ and all $j$, either $T\cap F_j=\varnothing$, or $T\cap F_j$ is regular with factors $(\ka_{r_1},\ldots,\ka_{r_n})$ and satisfies 
\begin{equation}
    |T\cap F_j|\sim \la.
\end{equation}
Thus, since each $x$ is $\alpha$-broad, we may apply a rescaled version of Proposition \ref{l2-refine-prop} to each $S_{\tau, j}f\Id_{T\cap F_j}$.
Writing $\ka=\ka_{r_1}\cdots\ka_{r_n}$, we obtain
\begin{equation}
\label{estimate-1}
    \|S_{\tau, j}f\Id_{T\cap F_j}\|_2^2\lesssim \ka^{-1}\Big(\frac{\la}{R^2\al}\Big)\|S_{\tau,j}f\|_2^2.
\end{equation}
Summing up all $\tau$ and $j$ in \eqref{model-operator-intro-2} to obtain the first estimate (see \eqref{5-16})
\begin{equation}
    \|S^\ast_R f\|_2^2\lesssim \kappa^{-1}\Big(\frac{\la}{R^2\al}\Big)\|f\|_2^2.
\end{equation}

\smallskip

Our second estimate handles the inefficiency of the broad-narrow analysis.
By further refinement, we may assume that $f$ is concentrated in a collection of $\al R$-balls $\cb$, and each $B\in\cb$ intersects $\sim \be$ many $\al R\times R$-rectangles from $\cup_\tau\ZT_\tau$.
Moreover, we may assume that for each $T\in\cup_\tau\ZT_\tau$, there are $\sim \nu$ many operators $\{S_{\tau,j}\}_j$ that make contribution to $S_{\tau, j}f\Id_{T\cap F_j}$.
Since $F_j$ are disjoint, we have $\lambda\nu\leq R^2\alpha$.
Then, with the help of the Kakeya maximal inequality (see Lemma \ref{kakeya-type-lem}) and a local $L^2$ estimate (Lemma \ref{local-L2}), we obtain our second estimate (see \eqref{second-L2})
\begin{equation}
\label{estimate-2}
     \|S^\ast_R f\|_2^2\lessapprox\Big(\frac{R^2\al}{\la\nu}\Big)\be^{-1}\big\|f\big\|_2^2.
\end{equation}
We refer to Proposition \ref{l2-prop2} for details.

\smallskip

The final estimate is given for the dual form of the model operator in the dual $L^4$-space. 
This is the most technical part of our proof (see Proposition \ref{l4/3-prop}).
To keep our introduction concise, we only present the final estimate and give a few remarks about it.
After duality, our third and final estimate reads (see \eqref{l4/3}) 
\begin{equation}
\label{estimate-3}
    \big\|S^\ast_R f\big\|_{4/3}^{4/3}\lessapprox \be^{1/3}\nu^{2/3}\ka^{2/3}\big\|f\big\|_{4/3}^{4/3}.
\end{equation}
The loss $\be^{1/3}$ comes from the inefficiency of the broad-narrow analysis.
It is sharp, by testing a function $f$ that is concentrated in an $R^{1/2}$-ball inside $B_R$, matching Tao's example.
The loss of $\nu^{2/3}$ comes from the lack of orthogonality in the radial direction.
Finally, the gain $\ka^{2/3}$ follows from \eqref{want-2} in Proposition \ref{l2-refine-prop}.
Here we cannot use the stronger estimate \eqref{want-1} since the property of $\al$-broadness is not preserved under duality.

\smallskip

Applying the estimate $\lambda \nu \leq R^2 \alpha$ and interpolating the estimates \eqref{estimate-1}, \eqref{estimate-2}, and \eqref{estimate-3} with weights $1/3$, $1/6$, and $1/2$, respectively, we obtain Theorem \ref{Lp-max}.
\medskip

\subsection{Further discussion on Proposition \ref{l2-refine-prop}.}

The method developed to prove Proposition \ref{l2-refine-prop} may be of independent interest. 
Additionally, we anticipate further applications of the refined $L^2$ estimate and its analogs. For instance, at a cost of $R^\e$, we can reprove the reverse square function estimate in $\mathbb{R}^2$ by solely working in the $L^2$ space and using $L^2$ orthogonality. The transition from $L^4$ space to $L^2$ space is facilitated by H\"older's inequality and the examination of level sets. This approach circumvents the reliance on the algebraic property of the number 4, namely, $4=2\times 2$, which was pivotal in its original proof in \cite{Cordoba-Kakeya}. Below, we will outline a sketch of its proof using refined $L^2$ estimates. 
It is our aspiration that this perspective will shed light on $L^p$-problems in harmonic analysis.


\begin{proof}[Sketch for the reverse square function estimate]

By employing the standard broad-narrow argument, we identify a dyadic number $\al\in[R^{-1/2}, 1]$ along with the corresponding $\alpha$-caps ${\tau}$ such that
\begin{equation}\label{1-24}
    \|Sf\|_4^4\lessapprox\sum_{\tau}\|S_\tau f\|_{L^4_{\Br(B_R)}}^4.
\end{equation}
Choosing $K=(\log R)^C$, for each $\tau$, via pigeonholing, we find a dyadic number $\lambda$ and a set $E_\lambda$ of unit balls such that for any unit ball
 $B\subset E_\la$, we have $\|S_\tau f\|_{L^4_{\Br(B)}}^4\sim \la$. Furthermore,
\begin{equation}
    \|S_\tau f\|_{L^4_{\Br(B_R)}}^4\lesssim(\log R)^C\|S_\tau f\|_{L^4_{\Br(E_\la)}}^4.
\end{equation}
Dividing $B_R$ into $R\alpha\times R$-rectangles ${T}$ with direction $\tau$, we have
\begin{equation}
    \|S_\tau f\|_{L^4_{\Br(E_\la)}}^4=\sum_T\|S_\tau f\|_{L^4_{\Br(E_\la\cap T)}}^4.
\end{equation}
Inside each $E_\la\cap T$, by reverse H\"older's inequality, we get
\begin{equation}
    \|S_\tau f\|_{L^4_{\Br(E_\la\cap T)}}^4\lesssim\frac{1}{|E_\la\cap T|}\|S_\tau f\|_{L^2_{\Br(E_\la\cap T)}}^4.
\end{equation}
Applying a rescaled version of \eqref{want-3} via parabolic rescaling, at a cost of $K^{O(1)} \lessapprox 1$, we observe that, after disregarding negligible Schwartz tails,
\begin{align}
    \|S_\tau f\|_{L^2_{\Br(E_\la\cap T)}}^4\lessapprox\frac{|E_\la\cap T|}{|T|}\|S_\tau f\|_{L^2(T)}^4\lessapprox\frac{|E_\la\cap T|}{|T|}\big(\int_T\sum_{\theta\subset\tau}|S_\theta f|^2\big)^2
\end{align}
by $L^2$-orthogonality inside $T$. 
Utilizing Cauchy-Schwarz's inequality and combining the above two estimates yield
\begin{equation}\label{1-29}
    \|S_\tau f\|_{L^4_{\Br(E_\la\cap T)}}^4\lessapprox\int_T\big(\sum_{\theta\subset\tau}|S_\theta f|^2\big)^2.
\end{equation}
Consequently, from (\ref{1-24}) and (\ref{1-29}), we obtain  
\begin{equation}
    \|Sf\|_4^4\lesssim\sum_{\tau}\|S_\tau f\|_{L^4_{\Br(B_R)}}^4\lessapprox\sum_\tau\int\big(\sum_{\theta\subset\tau}|S_\theta f|^2\big)^2\lessapprox\int\big(\sum_{\theta}|S_\theta f|^2\big)^2,
\end{equation}
which precisely corresponds to the reverse square function estimate.
\end{proof}

\bigskip

\noindent{\bf Notation: } 
We denote $A\lesssim B$ to signify that $A\leq CB$ for some constant $C$, and $A\lessapprox B$ to indicate that $A\leq C_\e R^\e B$ for any $\e>0$. 
We define $A\sim B$ if $A\lesssim B$ and $B\lesssim A$. 
In the paper, we consider a large number $R\gg1$. 
Another large number, denoted as $K$, satisfies $K\sim \exp{(\log R/\log\log R)}\lessapprox1$. 
The value of $K$ remains fixed starting from equation \eqref{before-K}.
$\theta$ (or $\theta'$) always represents an $R^{-1/2}$-cap.
$B_r$ stands for a ball in the plane with radius (or side length) $r>0$. 

\bigskip

\noindent
{\bf Acknowledgment.} Both authors are partially supported by the Simons Foundation and NSF2350101. 
The second author is partially supported by NSF2453583 and thanks Shengwen Gan for helpful discussions regarding the ending remark. 
The author also expresses gratitude to Ciprian Demeter for insightful comments that improved the presentation.

\bigskip

\section{Creating large regular subsets and the refined $L^2$ estimates}
\label{refined-l2-section}

In this section, our attention is directed towards investigating the weighted $L^2$ estimate $\|Sf\|_{L^2(E)}$, where $E$ represents a subset of an $R$-ball, denoted as $B_R$. 
We aim to prove Proposition \ref{l2-refine-prop} by constructing large regular subsets via a greedy algorithm.

\bigskip

\subsection{Algorithms for discovering extensive regular subsets}

Let us first state a lemma on finding significant regular subsets in a single scale.  
We will use this lemma repeatedly to obtain large regular subsets in multi-scales in the next lemma. 

\begin{lemma}
\label{regular-lem-single}
Suppose $E\subset B_r$ is a union of disjoint $r^{1/2}$-balls. 
Then there exists $\ka$ and a union of $r^{1/2}$-balls  $E_\ka\subset E$ with $|E_\ka|\gtrsim (\log r)^{-1}|E|$ and such that $E_\ka$ is $\ka$ regular in $B_r$.
\end{lemma}

\begin{proof}
We construct $E_\kappa$ via a greedy algorithm.

\smallskip 

\noindent
\textbf{Step 1: Greedy decomposition.}
Set $G_0 := \varnothing$ and $F_0 := E$. At stage $m$, assuming $F_m \neq \varnothing$, define
\begin{equation}
    \kappa_m^\ast := \sup_T \frac{|F_m \cap T|}{|T|},
\end{equation}
where the supremum is taken over all $r^{1/2} \times r$ tubes $T$.

Choose a dyadic number $\kappa_m \in [r^{-1/2},1]$ such that $\kappa_m \sim \kappa_m^\ast$, and select a tube $T_m$ satisfying
\begin{equation}
    |F_m \cap T_m| \sim \kappa_m |T_m|.
\end{equation}
Let $Q_m$ be the union of all $r^{1/2}$-balls $Q \subset F_m$ that intersect $T_m$. 
Then $Q_m \subset 2T_m$, where $CT$ is the rectangle concentric with $T$ with sides scaled by $C$, and
\begin{equation}
    |Q_m| \sim \kappa_m |T_m|.
\end{equation}
Indeed, the lower bound follows from $F_m \cap T_m \subset Q_m$, while the upper bound follows since $2 T_m$ can be covered by $O(1)$ translates of $T_m$, and $T_m$ maximizes $|F_m \cap T|$.

Now, we define
\begin{equation}
    G_{m+1} := G_m \sqcup Q_m, \qquad F_{m+1} := F_m \setminus Q_m.
\end{equation}
Continue this process until $|F_m| \le \tfrac12 |G_m|$. Since each step removes at least one $r^{1/2}$-ball, the process terminates after finitely many steps. 
At termination we obtain a set $G$ satisfying $|G| \ge \tfrac12 |E|$.

\medskip

\noindent
\textbf{Step 2: Dyadic pigeonholing.}
For each dyadic $\kappa \in [r^{-1/2},1]$, let $\mathbb T_\kappa$ be the collection of tubes $T_m$ selected at stages where $\kappa_m = \kappa$, and define
\begin{equation}
    E_\kappa := \bigsqcup_{m:\,\kappa_m = \kappa} Q_m.
\end{equation}
Then the sets $E_\kappa$ are disjoint and their union is $G$. 
Since there are $O(\log r)$ possible dyadic values of $\kappa$, there exists $\kappa$ such that
\begin{equation}
    |E_\kappa| \gtrsim (\log r)^{-1} |E|.
\end{equation}

\medskip

\noindent
\textbf{Step 3: Verification of regularity.}

\smallskip

\noindent
\emph{(1) Tube density bound.}
Let $m_0$ be the first stage such that $\kappa_{m_0} = \kappa$. Since the residual sets are nested, we have
\begin{equation}
    E_\kappa \subset F_{m_0}.
\end{equation}
Because $E_\kappa$ is a union of $r^{1/2}$-balls, the set $N_{r^{1/2}}(E_\kappa) \cap T$ is contained in the intersection of $F_{m_0}$ with $O(1)$ translates of $T$.
By the definition of $\kappa_{m_0}$, each such translate contains at most $O(\kappa |T|)$ mass of $F_{m_0}$. Hence, for every $r^{1/2} \times r$ tube $T$
\begin{equation}
    |N_{r^{1/2}}(E_\kappa) \cap T| \lesssim \kappa |T|.
\end{equation}

\smallskip

\noindent
\emph{(2) Covering by tubes.}
Since the sets $Q_m$ are disjoint and each satisfies $|Q_m| \sim \kappa |T_m| \sim \kappa r^{3/2}$, we obtain
\begin{equation}
    \# \mathbb T_\kappa \lesssim \frac{|E_\kappa|}{\kappa r^{3/2}}.
\end{equation}
Moreover, each $Q_m \subset 2 T_m$, so after enlarging each $T_m$ by a factor of 2, we obtain a family of tubes covering $E_\kappa$.

Finally, since $E_\kappa$ is a union of $r^{1/2}$-balls, we have $|N_{r^{1/2}}(E_\kappa)| \sim |E_\kappa|$.
Therefore,
\begin{equation}
    \# \mathbb T_\kappa \lesssim \frac{|N_{r^{1/2}}(E_\kappa)|}{\kappa r^{3/2}}.
\end{equation}
This verifies both conditions of Definition \ref{kappa-regular}, so $E_\kappa$ is $\kappa$-regular in $B_r$, completing the proof.
\qedhere

\end{proof}

\begin{lemma}
\label{regular-lem-multi}
Let $R\gg K\gg1$ with $K=R^{2^{-n}}$. Let $r_1<\cdots<r_n$ be scales such that $r_n=R$, $r_k=r_{k+1}^{1/2}$ for all $k\in \{1, \cdots, n-1\}$.  
Suppose that  $E\subset B_R$ is a union of finitely overlapping $K$-balls. 
Then there exists a collection of factors  $(\ka_{r_1},\ldots,\ka_{r_n})$ with each $\ka_{r_k}\in [r_k^{-1/2}, 1]$ and a union of $K$-balls   $E'\subset E$ with $|E'|\gtrsim (\log R)^{-3n}|E|$ such that $E'$ is regular with factors $(\ka_{r_1},\ldots,\ka_{r_n})$.
\end{lemma}
\begin{proof}
We create the set $E'$ inductively from the smallest scale $r_1=K^2$ to the largest one $r_n=R$ by introducing a sequence of sets $E_0,\ldots, E_n$.
Let $E_0=E$ and $r_0=K$. The algorithm starts at scale $r_1$, where Lemma \ref{regular-lem-single} starts to kick in.  
Let $k\in\{1,\cdots, n\}$. 
At the beginning of the $k$-th step of the algorithm, there is a union of $r_{k-1}$-balls $E_{k-1}$ such that $E_{k-1}$ is regular with factors $(\ka_{r_1},\ldots,\ka_{r_{k-1}})$ and $|E_{k-1}|\gtrsim(\log R)^{-3(k-1)}|E|$. 

\smallskip 

If $k=n$, then the algorithm terminates.  
Otherwise, consider a maximal covering of $r_k$-balls $\{Q\}$ of $E_{k-1}$.
At an $O(1)$ cost in the size of $E_{k-1}$, we may assume that these $r_k$-balls are disjoint.
Apply Lemma \ref{regular-lem-single} to the set $E_{k-1}\cap Q\subset Q$ so that there exists a factor $\ka_Q$ and a union of $r_{k-1}$-balls   $E_{\ka_Q}\subset E_{k-1}\cap Q$ such that 
\begin{enumerate}
    \item $|E_{\ka_Q}|\gtrsim(\log R)^{-1}|E_{k-1}\cap Q|$
    \item $E_{\ka_Q}$ is $\ka_Q$ regular in $Q$.
\end{enumerate}
By dyadic pigeonholing, there exists a collection of $r_k$-balls $\cQ$ such that 
\begin{enumerate}
    \item $\ka_Q$ are the same for all $Q\in\cQ$.
    \item $|N_{r_{k-1}}(E_{\ka_Q})\cap Q|$ are the same for all $Q\in\cQ$ up to a constant multiple.
    (Hence \eqref{uniform-assump} holds).
    \item $\big|E_{k-1}\cap \bigcup_{Q\in\cQ}Q\big|\gtrsim(\log R)^{-2}|E_{k-1}|$.
\end{enumerate}
We define
\begin{equation}
    E_k:
    = \bigcup_{Q\in\cQ} E_{\ka_Q}
\end{equation}
Then it is easy to see that $|E_k|\gtrsim(\log R)^{-3}|E_{k-1}|$, which yields $|E_{k}|\gtrsim(\log R)^{-3k}|E|$, through the inductive hypothesis. 
Let $\kappa_{r_k}=\kappa_Q$, so that $E_k$ is regular with factors $(\ka_{r_1},\ldots,\ka_{r_{k}})$ and $|E_{k}|\gtrsim(\log R)^{-3k}|E|$. 

\smallskip 

When the algorithm stops, we obtain set $E_n$ with $|E_{n}|\gtrsim(\log R)^{-3n}|E|$ such that $E_n$ is regular with factors $(\ka_{r_1},\ldots,\ka_{r_{n}})$. 
Take $E'=E_n$.
\qedhere
\end{proof}

\bigskip

\subsection{Refined $L^2$ estimates for a $\kappa$-regular set}

We now provide a quantitative characterization of the local $L^2$ norm of $Sf\Id_F$ where $F$ is a $\ka$-regular set. 
The argument is combinatorial in nature and relies on repeated applications of the pigeonhole principle.

\smallskip 

For a ball   $B$, we define the weight function $w_B$ to be
\begin{equation}
w_B(x) = \frac{1}{ \big( 1+ \frac{{\rm dist}(x, B)}{r(B)}\big)^{L}}\,,
\end{equation}
where $r(B)$ is the radius of the ball $B$, ${\rm dist}(x, B)$ represents the distance between $x$ and $B$, and  $L$ is a sufficiently large constant that may change from line to line. 
The weighted norm $\|f\|_{L^2(w_B)}$ is defined as  
\begin{equation}
    \|f\|_{L^2(w_B)} := \bigg( \int_{\mathbb R^2} \big|f\big|^2 w_B dx\bigg)^{1/2}\,.
\end{equation}
Similarly, the weighted broad norm $\|Sf\|_{L^2_{{\rm br}_M}(E)}$ is defined by substituting the corresponding $L^2(E)$-norm with the weighted norm $L^2(w_E)$ when $E$ is a ball.
The weight function $w_B$ is specifically crafted to address contributions from Schwartz tails. Essentially, it can be regarded as the indicator function $\Id_B$, since the primary contribution arises from $B$.

\smallskip 

Let $d(\tau)$ be the diameter (length) of a cap $\tau\subset\ZS^1$.
We will repeatedly use certain $L^2$ orthogonality and an essentially constant property, both arising from the wave packet decomposition of $Sf$, although the decomposition itself will not be used explicitly.
Let us summarize the into two lemmas.
\begin{lemma}
\label{l2-ortho-lem}
Let $Q$ be a ball of radius $r$.
Then we have
\begin{equation}
\label{l2-orthogonality}
    \big\|Sf\big\|_{L^2(Q)}^2\lesssim \sum_{\tau:\,d(\tau)=r^{-1}}\int_{2Q} |S_\tau f|^2.
\end{equation}
\end{lemma}
\begin{proof}
Let $\varphi_Q$ be a bump function equal to 1 on $Q$ and supported in $2Q$.
Thus,
\begin{equation}
    \big\|Sf\big\|_{L^2(Q)}^2\leq \int |Sf|^2\vp_Q\lesssim \sum_{\tau:\,d(\tau)=r^{-1}}\int |S_\tau f|^2\vp_{Q}\leq \sum_{\tau:\,d(\tau)=r^{-1}}\int_{2Q} |S_\tau f|^2.
\end{equation}
In the second inequality, we use the fact $\wh\vp_Q$ decays rapidly outside the $r^{-1}$-ball centered at the origin. 
This proves \eqref{l2-orthogonality}.
\end{proof}

\begin{lemma}
\label{essentially-constant}
Let $\tau\subset\ZS^1$ be an arc of diameter $r^{-1/2}$, and let $T$ be an $r^{1/2}\times r$-tube in direction $\tau$.
Let $\Id_T^\ast$ denote a smooth function that decays rapidly outside $T$, $\Id_T^\ast\gtrsim 1$ on $3T$, and $\widehat{\Id_T^\ast}$ is supported in an $r^{-1/2} \times r^{-1}$ cap centered at the origin, whose direction is normal to $\tau$.
Let $E$ be a set that is contained in $3T$.
Then
\begin{equation}
\label{schur-test}
    \int_{E}|S_\tau f|^2\lesssim \frac{|E|}{r^{3/2}}\int|\Id_T^\ast S_\tau f|^2.
\end{equation}
\end{lemma}
\begin{proof}
Note that $\Id_T^\ast S_\tau f=m_\tau\ast (\Id_T^\ast S_\tau f)$ for some multiplier $m_\tau$ such that $\|\Id_Q\ast m_\tau\|_\infty\lesssim |E|/r^{3/2}$.
Thus, \eqref{schur-test} follows from Young's convolution inequality and the simple pointwise estimate $|\Id_{E}S_\tau f|\lesssim |\Id_T^\ast S_\tau f|$.
\end{proof}


\begin{lemma}
\label{refined-l2-single-lem}
Let $F$ be a union of finitely overlapping $r^{1/2}$-balls $\{B\}$.
Suppose $F$ is $\ka_r$ regular in $B_r$. 
Then for any function 
$f\in L^2$, we have
\begin{equation}
\label{refined-l2-single-1}
    \sum_{B\subset F}\big\|Sf\big\|_{L^2_{\Br_M}(B)}^2\lesssim(\log R)^{5}\Big(\frac{|F|}{\ka_r |B_r|}\Big)\big\|Sf\big\|_{L^2_{\Br_{M-2}}(w_{B_r})}^2
\end{equation}
and
\begin{equation}
\label{refined-l2-single-2}
    \big\|Sf\big\|_{L^2(F)}^2\lesssim\ka_r\big\|Sf\big\|_{L^2(w_{B_r})}^2.
\end{equation}

\end{lemma}
\begin{proof} 
For $\tau\subset\ZS^1$ with $d(\tau)=r^{-1/2}$, denote by $\mathbb T_\tau$ a collection of finitely overlapping $r^{1/2}\times r$ tubes intersecting $B_r$ with direction $\tau$.
Recall that $\{B\}$ is a maximal covering of $F$ by $r^{1/2}$-balls.  
By Lemma \ref{l2-ortho-lem}, we have
\begin{equation}
    \big\|Sf\big\|_{L^2(F)}^2\lesssim\sum_{Q}\sum_{\tau:d(\tau)=r^{-1/2}}\int_{2Q} | S_\tau f|^2.
\end{equation}
Thus, by Definition \ref{kappa-regular} and Lemma \ref{essentially-constant} (with $E=2Q$), we get
\begin{align}
\label{no-broad}
    \big\|Sf\big\|_{L^2(F)}^2\lesssim\ka_r\sum_{\tau,d(\tau)=r^{-1/2}}\sum_{T\in\ZT_\tau}\int |\Id_T^\ast S_\tau f|^2\lesssim  \ka_r\big\|Sf\big\|_{L^2(w_{B_r})}^2,
\end{align}
which gives \eqref{refined-l2-single-2}. 

\smallskip 
 
Next, we prove \eqref{refined-l2-single-1}. 
Recall Definition \ref{kappa-regular}.
Since $F$ is $\ka$ regular, there is a set $\ZT_{\ka_r}$ of $r^{1/2}\times r$-tubes  such that $\#\ZT_{\ka_r}\lesssim|N_{r^{1/2}}(F)|/(\ka r^{3/2})\sim|F|/(\ka r^{3/2})$ and 
\begin{equation}
\label{partition-to-ball}
    \sum_{B\subset F}\big\|Sf\big\|_{L^2_{\Br_M}(B)}^2\lesssim\sum_{T\in\ZT_{\ka_r}}\sum_{B\subset F\cap T}\big\|Sf\big\|_{L^2_{\Br_M}(B)}^2.
\end{equation}
By dyadic pigeonholing on $\{\sum_{B\subset F\cap T}\big\|Sf\big\|_{L^2_{\Br_M}(B)}^2: T\in\ZT_{\ka_r}\}$, there exists a $\ZT_{\ka_r}'\subset\ZT_{\ka_r}$, such that
\begin{enumerate}
    \item $\sum_{B\subset F\cap T}\big\|Sf\big\|_{L^2_{\Br_M}(B)}^2$ are the same up to a constant multiple, for all $T\in\ZT_{\ka_r}'$.
    \item We have
    \begin{equation}
    \label{pigeonhole-l2-1}
        \sum_{T\in\ZT_{\ka_r}'}\sum_{B\subset F\cap T}\big\|Sf\big\|_{L^2_{\Br_M}(B)}^2\gtrsim(\log R)^{-1}\sum_{T\in\ZT_{\ka_r}}\sum_{B\subset F\cap T}\big\|Sf\big\|_{L^2_{\Br_M}(B)}^2.
    \end{equation}
\end{enumerate}

Fix an $r^{1/2}\times r$ tube $T\in\ZT_{\ka_r}$. By dyadic pigeonholing again, there exists a collection of $r^{1/2}$-balls $\cb(T)$ such that
\begin{enumerate}
     \item[$\bullet$] $\big\|Sf\big\|_{L^2_{\Br_M}(B)}^2$ are the same up to a constant multiple, for all $B\in\cb(T)$.
     \item[$\bullet$]  We have
    \begin{equation}
    \label{pigeonhole-l2-2}
        \sum_{B\in\cb(T)}\big\|Sf\big\|_{L^2_{\Br_M}(B)}^2\gtrsim(\log R)^{-1}\sum_{B\subset F\cap T}\big\|Sf\big\|_{L^2_{\Br_M}(B)}^2.
    \end{equation}
\end{enumerate}
Recall the definition of the broad norm in Definition \ref{def-broad-norm}. 
For each $r^{1/2}$-ball $B$, let $\Si(B)$ be the collection of $(\log R)^{-1}$-caps such that $\#\Si(B)=M$, and for any $\si\in\Si(B)$, we have
\begin{equation}
    \big\|Sf\big\|_{L^2_{\Br_M}(B)}^2\lesssim\big\|S_\si f\big\|_{L^2(B)}^2.
\end{equation}
For each $r^{-1/2}$-cap $\tau$, define a collection of $r^{1/2}\times r$-tubes $\ZT_{\tau}(B)$ as
\begin{equation}
    \ZT_{\tau}(B):=\{T\in \ZT_{\tau}:  T\cap 2B\not=\varnothing \}.
\end{equation}
Then, by Lemmas \ref{l2-ortho-lem} and \ref{essentially-constant} (with $E=2B$), 
\begin{align}
\label{l2-ortho-2}
    \big\|S_\si f\big\|_{L^2(B)}^2\lesssim\sum_{\tau\subset\si}\big\|S_\tau f\big\|_{L^2(2B)}^2\lesssim  r^{-1/2}\sum_{\tau\subset\si}\sum_{T\in\ZT_{\tau}(B)}\big\|\Id_{T}^\ast S_\tau f\big\|_{2}^2.
\end{align}
Here, the sum over $\tau$ ranges over a finitely overlapping cover of $\sigma$ by $r^{-1/2}$-caps.

For each $r^{1/2}\times r$-tube, let $\theta_T\subset \ZS^1$ be its directional cap. 
Denote by $\Si_{T}(B)\subset \Si(B)$ the collection of $(\log R)^{-1}$-caps $\sigma$ obeying $\dist(\theta_T,\si)\gtrsim (\log R)^{-1}$. Then $\# \Si_{T}(B)\geq M-2$. Let us consider the pairs $(\si,B)\in\Si\times\cb(T)$. Note that
\begin{equation}
\label{double-count-1}
    \sum_{B\in\cb(T)}\sum_{\si\in\Si}\Id_{\{\si\in\Si_{T}(B)\}}\geq (M-2)\#\cb(T).
\end{equation}
Since $\#\Si\lesssim\log R$, we have
\begin{equation}
    \text{L.H.S. of }\eqref{double-count-1}=\sum_{\si\in\Si}\sum_{B\in\cb(T)}\Id_{\{\si\in\Si_{T}(B)\}}\lesssim (\log R)\#\cb(T).
\end{equation}
Therefore, since for any $\si$, $\#\{B\in\cb(T):\si\in\Si_{T}(B)\}\leq\#\cb(T)$, there exists a set $\Si(T)\subset \Si$ such that 
\begin{enumerate}
    \item[$\bullet$] $\#\Si(T)\geq M-2$.
    \item[$\bullet$] For any $\si\in \Si(T)$, 
    \begin{equation}
    \label{number-B-lower-bound}
        \#\{B\in\cb(T):\si\in\Si_{T}(B)\}\gtrsim(\log R)^{-1}\#\cb(T).
    \end{equation}
\end{enumerate}
Since $\big\|Sf\big\|_{L^2_{\Br_M}(B)}^2$ are comparable up to a constant multiple for all $B\in\cb(T)$, it follows from \eqref{pigeonhole-l2-2} and \eqref{number-B-lower-bound} that, for any $\si\in\Si(T)$,
\begin{equation}
    \sum_{B\in\cb(T):\si\in\Si_T(B)}\big\|S_\si f\big\|_{L^2(B)}^2\gtrsim(\log R)^{-2}\sum_{B\in F\cap T}\big\|Sf\big\|_{L^2_{\Br_M}(B)}^2.
\end{equation}
Since $\dist(\theta_T, \si)\gtrsim(\log R)^{-1}$,  we see that any tube $T'\in \mathbb T_{\sigma, \tau}$ is transversal to the tube $T$. Thus, there are at most $O(\log R)$ many $r^{1/2}$-balls   $B$ belonging to $\mathcal B(T)\cap T'$.   
Thus, from \eqref{l2-ortho-2},  we get
\begin{align}
    \sum_{B\in\cb(T):\si\in\Si_T(B)}\big\|S_\si f\big\|_{L^2(B)}^2&\lesssim r^{-1/2}\sum_{\tau\subset\si} 
    \sum_{T'\in\ZT_{\tau}}\sum_{B\in \cb(T)\cap T'}\big\|S_\tau f\Id_{T'}\big\|_{L^2(B_r)}^2\\
    &\lesssim (\log R)r^{-1/2}\big\|S_\si f\big\|_{L^2(w_{B_r})}^2,
\end{align}
which yields that for any $\si\in\Si(T)$,
\begin{equation}\label{1*}
    \sum_{B\in F\cap T}\big\|Sf\big\|_{L^2_{\Br_M}(B)}^2\lesssim(\log R)^{3}r^{-1/2}\big\|S_\si f\big\|_{L^2(w_{B_r})}^2.
\end{equation}

\smallskip

Finally, consider the pairs $(\si,T)\in\Si\times\ZT_{\ka_r}'$. Since $\#\Si(T)\geq M-2$ for any $T\in\ZT_{\ka_r}'$, similarly, by pigeonholing, there exists a set $\Sigma'\subset \Si$ so that 
\begin{enumerate}
    \item[$\bullet$] $\#\Sigma'\geq M-2$.
    \item[$\bullet$] For any $\si\in \Sigma'$, 
    \begin{equation}
        \#\{T\in\ZT_{\ka_r}':\si\in\Si(T)\}\gtrsim(\log R)^{-1}\#\ZT_{\ka_r}'.
    \end{equation}
\end{enumerate}
 Since  $\sum_{B\subset F\cap T}\big\|Sf\big\|_{L^2_{\Br_M}(B)}^2$ are the same up to a constant multiple for all $T\in\ZT_{\ka_r}'$, by \eqref{pigeonhole-l2-1} and (\ref{1*}), we end up with, for any $\si\in\Sigma'$, 
\begin{align}
    \sum_{T\in\ZT_{\ka_r}}\sum_{B\in F\cap T}\big\|Sf\big\|_{L^2_{\Br_M}(B)}^2&\lesssim(\log R)^2\sum_{\substack{T\in\ZT_{\ka_r}'\\\si\in\Si(T)}}\sum_{B\in F\cap T}\big\|Sf\big\|_{L^2_{\Br_M}(B)}^2\\
    &\lesssim(\log R)^{5}r^{-1/2}(\#\ZT_{\ka_r})\big\|S_\si f\big\|_{L^2(w_{B_r})}^2.
\end{align}
Using  \eqref{partition-to-ball} and $\#\ZT_{\ka_r}\lesssim|N_{r^{1/2}}(E)|/(\ka r^{3/2})$,  we can finally conclude from the above estimate that for any $\si\in\Sigma'$, 
\begin{equation}
    \sum_{B\subset F}\big\|Sf\big\|_{L^2_{\Br_M}(B)}^2\lesssim(\log R)^{5}\Big(\frac{|F|}{\ka r^2}\Big)\big\|S_\si f\big\|_{L^2(w_{B_r})}^2,
\end{equation}
which yields, since $\#\Sigma'=M-2$, that
\begin{equation}
\label{yesbroad}
    \sum_{B\subset F}\big\|Sf\big\|_{L^2_{\Br_M}(B)}^2\lesssim(\log R)^{5}\Big(\frac{|F|}{\ka |B_r|}\Big)\big\|S_\si f\big\|_{L^2_{\Br_{M-2}}(w(B_r))}^2.
\end{equation}
Lemma \ref{refined-l2-single-lem} now follows from \eqref{no-broad} and \eqref{yesbroad}.
\end{proof}

\bigskip

\subsection{Proof of Proposition \ref{l2-refine-prop}}

We will use Lemma \ref{refined-l2-single-lem} at each scale, from the smallest one $r_1$ to the biggest one $r_n$. We can assume the weighted norm associated to $w_B$ in (\ref{refined-l2-single-1}) and \eqref{refined-l2-single-2} to be the norm associated to $B$, because the weighted function $w_B$ can be viewed essentially as $\Id_B$.  
In fact, by decomposing the whole space into dyadic annuli concentric with each ball $B$, we see that if the tail of the weight dominates for some $B$, the conclusion of the Proposition follows trivially.
Therefore, we may ignore the tail of the weight.

Denote $N_{r_k}(F)$ by $F_k$ for $k\in\{1, \cdots, n\}$. 
Since $F$ is regular with factors $(\ka_{r_1},\ldots, \ka_{r_n})$, for each $k$ we have
\begin{equation}
    \frac{|F_{k-1}\cap Q_{r_k}|}{|Q_{r_k}|}\sim\frac{|F_{k-1}|}{|F_k|}.
\end{equation}
for each $r_k$-ball $Q_{r_k}$ in a maximal disjoint covering of $F$ by $r_k$-balls.

At the $k$-th scale, apply \eqref{refined-l2-single-1} to each such  $r_k$-ball $Q_{r_k}$ and the set $F_{k-1}\cap Q_{r_k}$ to get 
\begin{align}
\nonumber
    \sum_{Q_{r_{k-1}}\subset F_{k-1}\cap Q_{r_k}}\!\!\!\big\|Sf\big\|_{L^2_{\Br_{M-2k+2}}(Q_{r_{k-1}})}^2
    &\lesssim(\log R)^{5}\Big(\frac{|F_{k-1}\cap Q_{r_k}|}{\ka_{r_k} |Q_{r_k}|}\Big)\big\|Sf\big\|_{L^2_{\Br_{M-2k}}(Q_{r_k})}^2\\ \nonumber
    &\lesssim(\log R)^{5}\Big(\frac{|F_{k-1}|}{\ka_{r_k} |F_k|}\Big)\big\|Sf\big\|_{L^2_{\Br_{M-2k}}(Q_{r_k})}^2.
\end{align}
Sum up all these balls $Q_{r_k}$ in $F_k$ so that
\begin{align}
\nonumber
    \sum_{Q_{r_{k-1}}\subset F_{k-1}}\big\|Sf\big\|_{L^2_{\Br_{M-2k+2}}(Q_{r_{k-1}})}^2
    \lesssim(\log R)^{5}\Big(\frac{|F_{k-1}|}{\ka_{r_k} |F_k|}\Big)\sum_{Q_{r_{k}}\subset F_{k}}\big\|Sf\big\|_{L^2_{\Br_{M-2k}}(Q_{r_k})}^2,
\end{align}
for every $k\in\{1,\cdots, n\}$. 
Multiplying these inequalities together results in obtaining \eqref{want-1} as desired.

Similarly, apply \eqref{refined-l2-single-2} for each $Q_{r_k}$ to get
\begin{equation}
    \big\|Sf\big\|_{L^2(F_{k-1}\cap Q_{r_k})}^2\lesssim\ka_{r_k}\big\|Sf\big\|_{L^2(Q_{r_k})}^2.
\end{equation}
Sum up all $Q_{r_k}$ in $N_{r_k}(F)$ so that
\begin{equation}
    \big\|Sf\big\|_{L^2(F_{k-1})}^2\lesssim\ka_{r_k}\big\|Sf\big\|_{L^2(F_k)}^2,
\end{equation}
for every $k\in\{1, \cdots, n\}$.
Consequently, \eqref{want-2} follows by taking product of these inequalities. Therefore, we complete the proof. 

\bigskip

\section{Exploring the maximal operator $T_*^\la$: Initial insights}
\label{linearization-section}

We are interested in the maximal operator $T_*^\lambda$ defined as in (\ref{br-mean}). It is clearly connected with the Fourier multiplier 
$$
 m_t^\la(\xi):= (2\pi)^{-n}(1-|\xi|^2/t^2)^\la_+\,, 
$$
for $\xi\in\mathbb R^n$. Going forward, we concentrate on the planar case with 
$n=2$,  while noting that many conclusions in this section remain applicable to higher-dimensional Euclidean spaces. 

\bigskip

\subsection{Reduction and linearization for the maximal operator}

We perform a standard frequency decomposition of the Fourier multiplier $m_t^\la$ by localizing the frequency space via a smooth partitioning of unity.  
Let $\eta_0$ be a bump function of   $B_{1/2}$, and $\eta_k$ be a bump function of the annulus $\{\xi\in\ZR^2:1-2^{-k}\leq |\xi|\leq 1-2^{-k-1} \}$ for $k\in \mathbb N$ such that $\sum_{k=0}^\infty \eta_k =1$ on the unit ball. 
We now define 
\begin{equation}
  m_{k, t}^\lambda (\xi)= m_t^\la(\xi) \eta_k(\xi/t)\,, 
\end{equation}
for any $\xi\in\mathbb R^2$. 
It is clear that $m_t^\lambda=\sum_{k=0}^\infty m^\la_{k,t}$.
Let $K_{k, t}^\lambda$ denote the inverse Fourier transform of $m_{k, t}^\la$, i.e., 
\begin{equation}
 \wh{K_{k, t}^\lambda} = m_{k, t}^\la\,. 
\end{equation}
By the triangle inequality, to prove Theorem \ref{main}, it suffices to prove that for each $k\geq 0$, 
\begin{equation}
\label{general-t}
    \big\|\sup_{t>0}|K_{k,t}^\la\ast f|\big\|_{p} \lesssim_\e 2^{k\e} 2^{-k\la}\big\|f\big\|_p
\end{equation}
for any $\e>0$ and $5/3\leq p \leq 2$. To further localize the variable $t$, we need the following proposition proved by Tao \cite{Tao-weak-type-BR}. A heuristic proof is presented in \cite{Tao-MBR}, Section 4. For a detailed exposition, see also the appendix of \cite{Li-Wu-MBR}. 

\begin{proposition}
\label{localization-prop}
Fix any $1\leq p\leq 2$, $k\in\ZN$, and $\la>0$. Suppose that
\begin{equation} 
    \big\|\sup_{t\in[1/2,1]}|K_{k,t}^\la\ast f|\big\|_{L^p(B_{2^k})}^p\lesssim_\e 2^{k\e}2^{-k\la}\big\|f\big\|_p^p.
\end{equation}
for any $\e>0$ and any $f\in L^p$. Then \eqref{general-t} is true. 
\end{proposition}

For a fixed $k$, let $R=2^k$.  Observe that 
\begin{equation}
    \wh{K_{k, t}^\la} =  2^{-k\la} a\big( R(t-|\xi|)\big)\,, 
\end{equation}
for some bump function on the annulus $\{\xi\in\ZR^2:t-R^{-1}\leq |\xi|\leq t+R^{-1}\}$. Consequently,  from the definition of $S_R^*$ in (\ref{Def-SR}) and Proposition \ref{localization-prop}, we deduce that \eqref{general-t} is equivalent to 
\begin{equation}
\label{local-t}
    \|S^\ast_R f\|_{L^p(B_R)}\lessapprox\|f\|_p,
\end{equation}
for $5/3\leq p\leq 2$, which precisely corresponds to (\ref{max-p-Sf}) stated in Theorem \ref{Lp-max}. \\

Next, we are going to linearize the maximal function $S_R^\ast f$. Let $\{t_j\}_{j=1}^{[R]}$ be a collection of $R^{-1}$-separated points in $[1,2]$. For $1\leq t_j\leq 2$, recall that $A_{t_j, R}$ is defined in Section \ref{Intro} as
the annulus with radius $t_j$ and thickness $R^{-1}$ given as follows: 
\begin{equation}
\label{aj}
    A_{t_j, R} =\big\{\xi\in\ZR^2:t_j-R^{-1}\leq |\xi|\leq t_j+R^{-1}\big\}.
\end{equation}
As in Section \ref{Intro}, $S_j$ ($=S_{t_j, R}$) is defined as 
\begin{equation}\label{defofS-j}
 \wh{S_j f}(\xi) = a\big( R(t_j-|\xi|) \big) \wh f(\xi)\,. 
\end{equation}

\begin{lemma}
\label{linearization-lem}
For any $1\leq p\leq 2$, there exists a collection of disjoint sets $\{F_j\}_{j=1}^{[R]}$, where each $F_j\subset B_R$ is a union of finitely overlapping unit balls, such that
\begin{equation}
    \|S_R^\ast f\|_p\lessapprox\big\|\sum_jS_jf\Id_{F_j}\big\|_p.
\end{equation}
\end{lemma}
\begin{proof}
Recall that $S_R^*f=\sup_{t\in [1, 2]}\big|S_{t, R}f\big|$.  
Consider $S_{t, R}f(x)$ as a function of $(x,t)$ in $\mathbb R^2\times \mathbb R$.  
Then its Fourier transform on $\mathbb R^2\times \mathbb R$ equals to $e^{-i\tau|\xi|}\wh f(\xi)\wh a(\tau/R)\Id_{B_2}(\xi)$, a function of $(\xi,\tau)$ that decays rapidly outside the region $B_2\times [-R,R]$.  
Thus, as a function of $(x,t)$, $\big|S_{t, R}f\big|$ is essentially constant in $B\times I$ (that is, $\|S_{t,R} f\|_\infty\sim_p \|S_{t,R}f\|_{L^p(w_{B\times I})}$ for all $p$), where  $B\subset B_R$ is any unit ball and $I\subset[1,2]$ is any $R^{-1}$ interval.
If the tail of the weight $w_{B\times I}$ in $\|S_{t,R}f\|_{L^p(w_{B\times I})}$ dominates, then there is nothing to prove. 
Otherwise,  we have
\begin{equation}
    \int\sup_{t\in I}|S_{t, R}f|^p\Id_{B} \sim \sup_{t\in I}\int|S_{t, R}f|^p\Id_{B}\lessapprox \int |S_j f|^p\Id_{B}\,,
\end{equation}
whenever $t_j\in I$.
Moreover, 
\begin{equation}
    \int\sup_{t\in [1,2]}|S_{t, R}f|^p\Id_{B}\sim \sup_{I\subset[1,2]}\int\sup_{t\in I}|S_{t, R}f|^p\Id_{B}.
\end{equation}

For each unit ball $B$, define $j_B$ as the minimal choice of $j\in \{1, \cdots, [R]\}$ so that 
\begin{equation}
    t_j\in I \text{ and } \sup_{t\in I}\int|S_{t, R}f|^p\Id_{B}\sim\int\sup_{t\in [1,2]}|S_{t, R}f|^p\Id_{B}.
\end{equation}
Now define $F_j$ as a union of unit balls $B$ with $j_B=j$. This gives
\begin{equation*}
    \int\sup_{t\in [1,2]}|S_{t, R}f|^p\lessapprox\int\sum_j|S_jf|^p\Id_{F_j}\,,
\end{equation*}
as desired.
\end{proof}

\bigskip

\subsection{An $L^2$ estimate localized to a single ball}

We aim to present a local $L^2$-estimate associated with a single ball. 
Before stating it, let us introduce some notation.  
Recall that $\tau$ stands for a cap defined as in Definition \ref{def1.7}. 
Let $\varphi_\tau$ be a bump function of the interval 
\begin{equation}
    \{\xi_1\in\mathbb R: N(\xi_1, \xi_2)\in\tau \}
\end{equation}
so that $\{\vp_\tau\}$ forms a smooth partition of unity of $[-1,1]$.
We define $S_{\tau, j}$ by 
\begin{equation}
\label{s-tau-j}
    \wh{S_{\tau,j}f}(\xi) = \wh{S_j f}(\xi)\varphi_\tau(\xi_1)\,,
\end{equation}
for $\xi=(\xi_1, \xi_2)$. The function $\varphi_\tau$ helps localizing smoothly the operator $S_jf$ on the cap $\tau$ in the frequency space.  
Thus, the Fourier transform of $S_{\tau, j}f$ is supported around a $1/R$-neighborhood of the cap $\tau$. 
The operator $S_{\tau, j}$ may vary from line to line depending on the change of the bump functions $a$ and $\varphi_\tau$, but such variation is harmless to our argument. 
For a ball $B$ in the plane, let $\psi_B$ be a bump function on $2B$ so that  $\psi_B(x)\sim 1$  for $x\in B$ and $|\psi_B^{(k)}|\lesssim_k r(B)^{-k}$ for any $k\in\mathbb N$.

\smallskip

We need to figure out what happens when the function $f$ is localized on a single ball  . 
This scenario can be addressed by the following lemmas,  the first of which is essentially the local $L^2$ estimate proved in \cite{Tao-MBR}.

\begin{lemma}
\label{local-L2-0}
Let $B$ be an $R\alpha$-ball with $0\leq \alpha\leq 1$. 
For any positive integer $j$ and $f_j\in L^2$,   
\begin{equation}
\label{local-l2-2}
    \big\|\sum_{j}S_{j}(f_j)\psi_B\big\|_{2}^2\lessapprox\al\big\|\sum_{j}|f_j|\big\|_2^2.
\end{equation}
\end{lemma}

\begin{proof} 
We will utilize the $TT^*$-method to prove Lemma \ref{local-L2-0}. 
Via the method of stationary phase, the kernel of the operator $S_j$, denoted by $K_j$, exhibits the following asymptotic behavior (see for instance \cite{Stein-Harmonic-Analysis} Chapter IX)
\begin{equation}
    K_j (x)= \varphi(R^{-1}x)R^{-1}|x|^{-1/2}\big(e^{it_j|x|}\sum_{j=0}^{\infty}a_j|x|^{-j}+e^{-it_j|x|}\sum_{j=0}^\infty b_j|x|^{-j}\big),
\end{equation}
where $\varphi$ is a smooth function, $a_j$ and $b_j$ are constants depending on $j$.  The principal contribution arises from 
the first term when $j=0$.  
Thus, it suffices to prove 
\begin{equation}
\label{local-l2-02}
    \big\|\sum_{j}S'_{j}(f_j)\psi_B\big\|_{2}^2\lessapprox\al\big\|\sum_{j}|f_j|\big\|_2^2,
\end{equation}
where $S'_j f= K'_j*f$ with 
\begin{equation}
\label{K-j'}
    K'_j(x) = \varphi(R^{-1}x)R^{-1}|x|^{-1/2} e^{it_j|x|}\,. 
\end{equation}

For a dyadic number $1\leq r\lesssim R$, let $\eta_r(x)$ be a bump function of the dyadic annulus $\{x: |x|\sim r\}$, and let $\eta_0$ be a bump function of the unit ball so that $\{\eta_0, \eta_r: r\text{ dyadic}\}$ forms a partition of unity of the whole space.
Again, the bump functions $\{\eta_r\}$ may vary from line to line, but they remain harmless to our argument.
Thus,
\begin{equation}
    K'_j=\eta_0 K'_j+\sum_{r} \eta_r K'_j. 
\end{equation}
For each $r$, define 
\begin{equation}
\label{vp-j-r}
    K_{j,r}(x):=\eta_r(x) K_j'(x) \,,
\end{equation}
which, by \eqref{K-j'}, is equal to $ R^{-1}r^{-1/2}\eta_r(x)e^{it_j|x|}$. 
Moreover, we define $S_{j,r}f:=K_{j,r}\ast f$. 
Since $K_{j,r}$ is supported in  $\{x:|x|\sim r\}$, by a standard localization argument (see \cite{Wu-BR} Lemma 2.2), it suffices to prove that for any $r$-ball $B_r$,
\begin{equation}\label{local-l2-3}
\big\|\sum_{j}S_{j,r}(f_j\Id_{B_r})\psi_B\big\|_{2}^2\lesssim\al\big\|\sum_{j}|f_j|\Id_{B_r}\big\|_2^2.
\end{equation}

\smallskip
Without loss of generality,  we can confine the variable $x$ in \eqref{vp-j-r} to the conic region  $\{x:0\leq x/|x|\leq \pi/4\}$.
For any fixed $y_1\in[-R,R]$, we define $g_{j,y_1}(\cdot)=|f_j(y_1,\cdot)|\Id_{B_r}(y_1,\cdot)$.
Hence, by freezing the first variable $y_1$ and Cauchy-Schwarz, it suffices to prove for all $y_1$,
\begin{equation}
\label{freezing}
    \big\|\sum_{j}S_{\tau, j}(g_{j,y_1})\psi_B\big\|_{2}^2\lesssim \al r^{-1}\big\|\sum_{j}g_{j, y_1}\big\|_2^2.
\end{equation}
Assume $y_1=0$ without loss of generality.
Apply \eqref{vp-j-r} so that the left-hand side of \eqref{freezing} equals to 
\begin{equation}
    \sum_{j_1,j_2}\int_{\ZR^2}g_{j_1}(z_1)g_{j_2}(z_2)K_{\tau,j_1,j_2}(z_1,z_2)dz_1dz_2\,,
\end{equation}
where $K_{j_1,j_2}(z_1,z_2)$ is defined by
\begin{equation}
    K_{j_1,j_2}(z_1,z_2):= \frac{1}{rR^2}\int e^{i(t_{j_1}|x-\bar z_1|-t_{j_2}|x-\bar z_2|)}\psi_B^2(x)a_{j_1,j_2}(x,z_1,z_2)dx\,.
\end{equation}
Here $\bar z_1=(0,z_1)$ and $\bar z_2=(0,z_2)$, and $a_{j_1,j_2}$ is a suitable amplitude function that vanishes unless $|x-\bar z_1|, |x-\bar z_2|\sim R$.

Since $|x|\sim r$ and since $0\leq x/|x|\leq \pi/4$, the first coordinate of $x$ has size $\sim r$.
Thus, by elementary geometry, we have $\nabla_x(t_{j_1}|x-\bar z_1|-t_{j_2}|x-\bar z_2|)\gtrsim r^{-1}|z_1-z_2|$.
Since the $N$-th derivative of the amplitude function $\psi_B^2a_{j_1,j_2}$ is $O_N(\min\{r,|B|^{1/2}\}^{-N})$, from the integration by parts, we end up with 
\begin{align}
    |K_{\tau,j_1,j_2}(z_1,z_2)|&\lesssim_N \frac{1}{rR^2}\min\{|B|,r^2\}(1+(r^{-1}|B|^{1/2}+1)|z_1-z_2|)^{-N}\\
    &=\min\{\frac{\al^2}{r},\frac{r}{R^2}\}(1+(\al(R/r)+1)|z_1-z_2|)^{-N}.
\end{align}
This yields $\int|K_{j_1,j_2}(z_1,z_2)|dz_1, \int |K_{j_1,j_2}(z_1,z_2)|dz_2\lesssim\min\{\frac{\al^2}{r},\frac{\al}{R}\}$, which concludes \eqref{freezing} by Schur test since $r\leq R$ and $\al\leq 1$. 
This establishes \eqref{local-l2-3} and consequently \eqref{local-l2-2}, as desired. 
\end{proof}

\begin{lemma}
\label{local-L2}
Let $\al\in[R^{-1/2},1]$. Suppose that $B$ is an $R\al$-ball and $\{\tau\}$ is a collection of finitely overlapping $\al$-caps. Additionally, assume that $\{E_{\tau,j}\}_{\tau,j}$ is a collection of disjoint sets
in the plane. 
Then, for any $f\in L^2$,
\begin{equation}
    \sum_{\tau}\sum_{j}\big\|S_{\tau, j}(f\Id_{ B})\Id_{E_{\tau,j}}\big\|_2^2\lessapprox \al\big\|f\Id_B\big\|_2^2.
\end{equation}

\end{lemma}
\begin{proof}
 By duality, it suffices to show that, for any $f\in L^2$, 
\begin{equation}
\label{dual-L2}
    \big\|\sum_\tau\sum_{j}S_{\tau, j}(f\Id_{E_{\tau,j}})\big\|_{L^2(B)}^2\lessapprox\al\sum_{\tau}\sum_j\big\|f\Id_{E_{\tau,j}}\big\|_2^2.
\end{equation}
Since $\al\geq R^{-1/2}$ and since the Fourier transform of $\sum_{j}S_{\tau, j}(f\Id_{E_{\tau,j}})$ is supported in an $\al\times 1$-strip with direction $\tau$, similar to Lemma \ref{l2-ortho-lem}, we have
\begin{equation}
    \big\|\sum_\tau\sum_{j}S_{\tau, j}(f\Id_{E_{\tau,j}})\big\|_{L^2(B)}^2\lesssim\sum_\tau\big\|\sum_{j}S_{\tau, j}(f\Id_{E_{\tau,j}})\big\|_{L^2(2B)}^2.
\end{equation}
Hence, to prove \eqref{dual-L2}, we only need to prove for a single $\tau$,
\begin{equation}
\label{dual-L2-single}
    \big\|\sum_{j}S_{\tau, j}(f\Id_{E_{j}})\big\|_{L^2(B)}^2\lessapprox\al\sum_{j}\big\|f\Id_{E_{j}}\big\|_2^2,
\end{equation}
where $\{E_j\}_j$ is a collection of disjoint sets. It is more convenient to replace $\Id_B$ by a smooth cut-off $\psi_B$, 
so we will demonstrate the following slightly stronger inequality: 
\begin{equation}
\label{local-l2-1}
    \big\|\sum_{j}S_{\tau, j}(f_j)\psi_B\big\|_{2}^2\lessapprox\al\sum_{j}\big\|f_j\big\|_2^2,
\end{equation}
where $f_j\in L^2$ and  the supports of $f_j$'s are disjoint.   

\smallskip 

Observe that 
\begin{equation}
\label{s-tau-j-tilde}
    S_{\tau,j}f=S_{j}(\wt S_{\tau,j }f),
\end{equation}
where ${\wt S}_{\tau, j}f$ can be represented as $\wt K_{\tau, j}*f$ such that 
the Fourier transform of the kernel $\wt K_{\tau, j}$ is a smooth bump function supported in an $2\alpha$-ball containing the cap 
$\tau$. Thus, the kernel $\wt K_{\tau, j}$ obeys the decaying estimate,
\begin{equation}
\label{k-tau-j-tilde}
    |\wt K_{\tau,j}(x)|\lesssim_N \al^2(1+\al |x|)^{-N},
\end{equation}
which yields that $\sup_{\tau, j}\big|\wt K_{\tau, j}| \in L^1$.

\smallskip

Apply Lemma \ref{local-L2-0} with $f_j=\wt S_{\tau,j }(f\Id_{E_j})$ so that
\begin{equation}
    \big\|\sum_{j}S_{\tau, j}(f\Id_{E_j})\psi_B\big\|_{2}^2\lessapprox\al\big\|\sum_{j}|\wt S_{\tau,j }(f\Id_{E_j})|\big\|_2^2.
\end{equation}
Thus, to prove \eqref{local-l2-1}, we only need to show
\begin{equation}
    \big\|\sum_{j}|\wt S_{\tau,j }(f_j)|\big\|_2^2\lesssim\sum_{j}\big\|f_j\big\|_2^2\,,
\end{equation}
for $f_j$'s with disjoint supports.  
In fact, because $\sup_{\tau,j}|\wt K_{\tau,j}|$ is integrable via \eqref{k-tau-j-tilde},  
\begin{align}
    \big\|\sum_{j}|\wt S_{\tau,j }(f_j)|\big\|_2^2&\lesssim \big\|(\sup_{\tau,j}|\wt K_{\tau,j}|)\ast \sum_{j}|f_j|\big\|_2^2\\
    &\lesssim\big\|\sum_{j}|f_j|\big\|_2^2=\sum_{j}\big\|f_j\big\|_2^2.
\end{align}
In the last equality, we use the fact that the supports of $f_j$'s are disjoint. 
This gives \eqref{local-l2-1}, and therefore
we complete the proof of Lemma \ref{local-L2}. 
\end{proof}

\begin{remark}
\rm

Another possible approach to prove \eqref{local-l2-1} is as follows: We first derive the explicit expression of the kernel of $S_{\tau,j}$ using the method of stationary phase. Then, we conclude \eqref{local-l2-1} via a similar idea in \cite{Tao-MBR}. However, the explicit expression of the kernel of $S_{\tau,j}$ is not straightforward since it depends on the bump function 
$\varphi_\tau$. We anticipate that proving \eqref{local-l2-1} by using this method would require more effort.

\end{remark}

\bigskip

\subsection{$L^4$ orthogonality}

In this subsection, we examine the orthogonality in $L^4$ concerning functions whose Fourier transforms are distributed around a neighborhood of the circle $ \{\xi\in\mathbb R^2: |\xi|=1\}$. Recall that $\{t_1, t_2, \cdots, t_{[R]}\}$ is a collection of $1/R$-separated real numbers in $[1, 2]$.  For given $j\in [1, R]\cap\mathbb Z$ and a cap $\tau$,  let $f_{\tau, j}$ be a function whose Fourier transform is supported in  
$C_\tau \cap A_{t_j, R}$, where $C_\tau$ is the conic region of the cap $\tau$ defined as in  (\ref{Def1.6})
and $A_{t_j, R}$ is the annulus given by (\ref{aj}).  We use $\mathcal J_\tau$ to denote a subset of 
$ [1, R]\cap \mathbb Z$ which may depend on $\tau$. For any interval $I$ in $[1/2, 2]$ and any cap $\tau$, we define
\begin{equation}
\cj_{\tau, I}=\{j\in\cj_\tau:t_j\in I\} \,. 
\end{equation}

\begin{lemma}
\label{l4-ortho-lem-1}
Let $\al<\bar\alpha$ be two numbers in $[R^{-1/2},1]$, and $\tau_0,\tau_0'$ be two $\bar\alpha$-caps with 
${\rm dist}(\tau_0, \tau'_0) \lessapprox \bar\alpha$.
Given an $R\al$-ball $B$, let $\cT_B,\cT_B'$ be two collections of finitely overlapping $\al$-caps with $\#\cT_B, \#\cT_B'\lesssim \mu_1$ such that $\tau\subset\tau_0$, $\tau'\subset\tau_0'$ for any $\tau\in\cT_B$, $\tau'\in\cT_B'$. 
Let $\mathcal I= \{I\}$ be a finitely overlapping $\al\bar\alpha^{-1}$-intervals in $[1/2,2]$.  
Suppose also that there is a number $\mu_2$ such that for each $\tau\in\cT_B\cup\cT_B'$, $\#\{I\in \mathcal I:\cj_{\tau,I}\not=\varnothing\}\lesssim \mu_2$. 
Then for any $f\in L^2$,
\begin{align}
    &\int_B\big|\sum_{\tau\in\cT_B}\sum_{j\in\cj_{\tau}}f_{\tau,j}\big|^2\big|\sum_{\tau'\in\cT_B'}\sum_{j'\in\cj_{\tau'}}f_{\tau',j'}\big|^2\\
    \lessapprox &\, \mu_1\mu_2\sum_{\tau\in\cT_B}\sum_{\tau'\in\cT_B'}\sum_{I,I'\in\mathcal I}\int_{2B}\big|\sum_{j\in\cj_{\tau,I}}f_{\tau,j}\big|^2\big|\sum_{j'\in\cj_{\tau',I'}}f_{\tau',j'} \big|^2.
\end{align}

\end{lemma}
\begin{proof}
Let us define $F_{\tau,I}$ as
\begin{equation}
    F_{\tau,I}:=\sum_{j\in\cj_{\tau,I}}f_{\tau,j}.
\end{equation}
Thus, it suffices to show that 
\begin{align}
\label{before-plancherel}
    \int_B\big|\sum_{\tau\in\cT_B}\sum_{I\in\mathcal I}F_{\tau,I}\big|^2\big|\sum_{\tau'\in\cT_B'}\sum_{I'\in \mathcal I}F_{\tau',I'}\big|^2\lessapprox\mu_1\mu_2\sum_{\tau,\tau'}\sum_{I,I'}\int_{2B}|F_{\tau,I}|^2|F_{\tau',I'}|^2.
\end{align}
Observe that  $F_{\tau, I}$ has its Fourier support in $C_\tau \cap \bigcup_{j\in I} A_{t_j, R}$,  which is a subset of an $\al\bar\alpha^{-1}\times\al$-rectangle, denoted by $R_{\tau,I}$.

\smallskip

Let $\psi_B$ be a Schwartz function whose Fourier transform decays rapidly outside the ball of radius $(R\alpha)^{-1}$ centered at the origin, such that $\psi_B(x) \geq 1$ for all $x \in B$, and supported on $2B$.
Denote the left-hand side of \eqref{before-plancherel} by $\Lambda(\tau, \tau')$.
Thus,
\begin{align}
\label{before-CS}
    \Lambda(\tau,\tau')&\lesssim\int\big|\sum_{\tau\in\cT_B}\sum_{I}F_{\tau,I}\psi_B\big|^2\big|\sum_{\tau'\in\cT_B'}\sum_{I'}F_{\tau',I'}\psi_B\big|^2\\ \nonumber
    \lesssim\sum_{\tau_1,\tau_2,\tau_1',\tau_2'}\sum_{I_1,I_2,I_1',I_2'}\int& (\wh F_{\tau_1, I_1}\ast\wh\psi_B) \ast\overline{(\wh F_{\tau_1', I_1'}\ast\wh\psi_B)}\cdot \overline{(\wh F_{\tau_2, I_2}\ast\wh\psi_B)} \ast(\wh F_{\tau_2', I_2'}\ast\wh\psi_B),
\end{align}
where the sum is taken over all $(\tau_1, I_1), (\tau_1', I_1'), (\tau_2, I_2), (\tau_2', I_2')$ such that
\begin{equation}
\label{quadruples}
    (R_{\tau_1,I_1}+R_{\tau_1',I_1'})\cap (R_{\tau_2,I_2}+R_{\tau_2',I_2'})\not=\varnothing.
\end{equation}
Thus, by applying the Cauchy--Schwarz inequality to the integral on the right-hand side of \eqref{before-CS},
\begin{align}
    \Lambda(\tau,\tau')\lesssim&\sum_{\tau_1,I_1,\tau_1', I_1'}\bigg( \int|(\wh F_{\tau_1, I_1}\ast\wh\psi_B) \ast\overline{(\wh F_{\tau_1', I_1'}\ast\wh\psi_B)}|^2 \bigg)^{1/2}\\ 
    \label{sum-tau-I-2}
    &\cdot  \sum_{\tau_2,I_2,\tau_2', I_2'}\bigg( \int|(\wh F_{\tau_2, I_2}\ast\wh\psi_B) \ast\overline{(\wh F_{\tau_2', I_2'}\ast\wh\psi_B)}|^2\bigg)^{1/2},
\end{align}
where, for fixed $(\tau_1, I_1, \tau_1', I_1')$, the sum over $(\tau_2, I_2, \tau_2', I_2')$ in \eqref{sum-tau-I-2} is taken subject to the condition \eqref{quadruples}.
Because each $\tau$ and $\tau'$ lie in a $ O(\bar\alpha)$-conic region, the rectangles $R_{\tau_1, I_1}, R_{\tau_1', I'_1},R_{\tau_2, I_2}, R_{\tau_2', I'_2}$ can be essentially viewed as rectangles pointing in one direction. 
Hence, we see that  for a fixed quadruple $(\tau_1,\tau_1',I_1,I_1')$, there are $\lesssim  \mu_1\mu_2$ quadruples $(\tau_2,\tau_2',I_2,I_2')$ such that (\ref{quadruples}) holds.  
Since $\wh\psi_B$  decays rapidly outside the ball of radius $(R\alpha)^{-1}$ centered at the origin, we end up with
\begin{align}
   \Lambda(\tau,\tau')&\lessapprox \mu_1\mu_2\sum_{\tau,I,\tau', I'}\int|(\wh F_{\tau, I}\ast\wh\psi_B) \ast\overline{(\wh F_{\tau', I'}\ast\wh\psi_B)}|^2  \\  &\lessapprox\mu_1\mu_2\sum_{\tau,\tau'}\sum_{I,I'}\int_{2B}|F_{\tau,I}|^2|F_{\tau',I'}|^2\,.
\end{align}
This establishes \eqref{before-plancherel}.
\end{proof}

The proof of the other lemma is based on a simple geometric observation, as noted in \cite{Carbery-MBR}. 
A similar result can also be found in \cite[Section 10]{MBR-R3}.

\begin{lemma}
\label{l4-ortho-lem-2}
Given $\al\in[R^{-1/2},1]$, let $\tau, \tau'$ be two $\al$-caps and $j, j'\in [1, R]\cap \mathbb Z$. 
Let $\theta$ and $\theta'$ be $R^{-1/2}$-caps forming a finitely overlapping partition for $\tau$ and $\tau'$, respectively.
Suppose $f_{\tau, j} $ and $f_{\tau', j'} $ are two $L^2$ functions such that
\begin{equation}
    f_{\tau, j} = \sum_{\theta} f_{\theta, j}, \qquad
    f_{\tau', j'} = \sum_{\theta'} f_{\theta', j'},
\end{equation}
where $\widehat{f}_{\theta, j}$ is supported in the $R^{-1/2} \times R^{-1}$ rectangle $C_\theta \cap A_{t_j, R}$, and $\widehat{f}_{\theta', j'}$ is supported in the $R^{-1/2} \times R^{-1}$ rectangle $C_{\theta'} \cap A_{t_{j'}, R}$.
Then
\begin{equation}
    \int|f_{\tau,j}|^2 | f_{\tau',j'}|^2 dx \lesssim \sum_{\theta\subset\tau}\sum_{\theta'\subset\tau'}\int|f_{\theta,j}|^2 | f_{\theta',j'}|^2 dx.
\end{equation} 
\end{lemma}

\begin{proof}
By the Plancherel theorem, we have 
\begin{equation}\label{pla}
    \int|f_{\tau,j}|^2 | f_{\tau',j'}|^2dx =  \sum_{\theta_1, \theta_1', \theta_2, \theta_2'} \int \wh{f_{\theta_1, j}}*\overline{\wh {f_{\theta_1', j'}}}\cdot\overline{ \wh{f_{\theta_2, j}}}*\wh{f_{\theta'_2, j'}} \,,
\end{equation}
where $\theta_1, \theta_2\subset \tau$ and $\theta'_1, \theta_2'\subset \tau'$ are $R^{-1/2}$-caps.
For $k\in\{1, 2\}$, the function $\wh{f_{\theta_k, j}}*\overline{\wh {f_{\theta_k', j'}}}$ is supported in 
\begin{equation}
    \supp(\wh{f_{\theta_k, j}}) + \supp(\wh {f_{\theta_k', j'}})\,.
\end{equation}
Note that $\wh f_{\theta, j}$ is supported in $C_\theta\cap A_{t_j, R}$, which is essentially a rectangle of dimensions $R^{-1}\times R^{-1/2}$. 
The key geometric observation is that the sets  $\supp(\wh{f_{\theta_k, j}}) + \supp(\wh {f_{\theta_k', j'}})$ are finitely overlapping for all $\theta_k\subset \tau$ and all $\theta_k'\subset \tau'$ (see \cite{Carbery-MBR}, or \cite[Section 10]{MBR-R3}).
Consequently, combining this geometric observation with (\ref{pla}), we derive that
\begin{equation}
    \int|f_{\tau,j}|^2 | f_{\tau',j'}|^2dx\lesssim \int \sup_{\theta, \theta'} \big|\wh{f_{\theta, j}}*\wh {f_{\theta', j'}}\big|^2\,,
\end{equation}
which is bounded above by 
\begin{equation}
    \sum_{\theta, \theta'} \int \big|\wh{f_{\theta, j}}*\wh {f_{\theta', j'}}\big|^2=  \sum_{\theta, \theta'} \int  |f_{\theta,j}|^2 | f_{\theta',j'}|^2 dx\,. \qedhere
\end{equation}

\end{proof}

\section{Analyzing the model operator: A thorough investigation}
\label{Broad-narrow-section}

Lemma \ref{linearization-lem} clarifies that our focus can be narrowed down to examining the model operator defined as
\begin{equation}
\label{model-operator}
\sum_jS_jf\Id_{F_j},
\end{equation}
with $S_j=S_{t_j, R}$, $\{F_j\}$ comprising disjoint subsets in the plane, and $F_j\subset B_R$ forming unions of unit balls. 
To verify Theorem \ref{Lp-max}, it suffices to demonstrate that
\begin{theorem}
\label{main-thm}
For $p=5/3$ and any $f\in L^p(\ZR^2)$, we have 
\begin{equation}
\label{main-esti}
    \big\|\sum_jS_jf\Id_{F_j}\big\|_p\lessapprox\big\|f\big\|_p.
\end{equation}
\end{theorem}

\noindent Theorem \ref{main} can be derived from Theorem \ref{Lp-max} via Proposition \ref{localization-prop}. Consequently, our focus now shifts to establishing Theorem \ref{main-thm}.

\medskip

We initiate the proof of Theorem \ref{main-thm} by employing a broad-narrow reduction technique on our model operator \eqref{model-operator}, reminiscent of the broad-narrow analysis method used in the work 
of Bourgain and Guth  in \cite{Bourgain-Guth-Oscillatory} and Guth in \cite{Guth}.

\bigskip

\subsection{The broad-narrow analysis}

\begin{definition}
For any $\rho$-cap $\tau$, we define $\tau^*$ to be the $2(\log R) \rho$-cap that is concentric with $\tau$.
\end{definition}

\begin{lemma}
\label{broad-narrow}
Fix $j$. 
For each $x\in F_j$, either one of the following statements is true.
\begin{enumerate}
    \item [$\bullet$] $x$ is {\bf{$\al$-broad}} for some dyadic number $R^{-1/2}<\al\leq 1$. This means that there exist an $\al$-cap $\tau$ and 
    at least $10\log\log R$ many $\alpha$-caps $\tau'$ contained in $\tau^*$ such that 
    \begin{equation}
        |S_{\tau',j}f(x)|\gtrapprox|S_{\tau,j}f(x)|\gtrapprox |S_jf(x)|.
    \end{equation}
    
    \item[$\bullet$] $x$ is {\bf{narrow}}. This means that there exists an $R^{-1/2}$-cap $\theta$ so that
    \begin{equation}
        |S_{\theta,j}f(x)|\gtrapprox|S_jf(x)|.
    \end{equation}
\end{enumerate}
\end{lemma}
\begin{proof}
Without loss of generality, we can assume that $\log R$ is a dyadic number.
Let us  introduce $\lesssim \log R/\log\log R$ many scales $\rho_0,\rho_1\,\cdots,\rho_n$, where $\rho_k=(\log R)^{-k}$ for $k\in\{0, \cdots, n-1\}$ and $\rho_n\approx R^{-1/2}$.
For each $\rho_k$, partition the unit circle into disjoint $\rho_k$-caps $\{\om_k\}$. 
We establish the lemma by running the following algorithm. 

\medskip

First we define $S_{\om_0,j}f=S_j f$. The algorithm starts from the first scale $\rho_1$, which we designate as the first step. In the $k$-th step, there is a function $S_{\om_{k-1},j}f$ from the $(k-1)$-th step whose Fourier transform is  supported in $A_{t_j, R}\cap C_{\om_{k-1}}$.  We partition the multiplier operator $S_{\om_{k-1},j}$ using $\rho_k$-caps $\{\om_k\}$ so that 
\begin{equation}
    S_{\om_{k-1},j}f=\sum_{\om_k\subset 2\om_{k-1}}S_{\om_{k},j}f.
\end{equation}
If there exists a $\rho_k$-cap $\om_k$ and at least $10\log\log R$ many $\rho_k$-caps $\omega'_k$ contained in $\om_k^*$ such that 
\begin{equation}
\label{dichotomy-1}
    |S_{\om_k',j}f(x)|\gtrsim|S_{\om_k,j}f(x)|\geq (10\log R)^{-1} |S_{\om_{k-1},j}f(x)|,
\end{equation}
then we terminate the algorithm by setting  $\al=\rho_k$, $\tau'=\om_k'$, and $\tau=\om_k$.

\medskip

Otherwise, by the triangle inequality and pigeonholing, there exists a $\rho_k$-cap $\om_k$ satisfying
\begin{equation}
\label{dichotomy-2}
    |S_{\om_{k},j}f(x)|\gtrsim (\log\log R)^{-1}|S_{\om_{k-1},j}f(x)|.
\end{equation}
We then proceed with the algorithm to the $(k+1)$-step.

\medskip

If the algorithm halts at some $k<n$, then \eqref{dichotomy-1} satisfies the first part of the lemma, as $k<n\lesssim {\log R}/{\log\log R}$ and due to \eqref{dichotomy-2},
confirming the $\alpha$-broadness of $x$. 
Otherwise, for a similar rationale, there exists an $R^{-1/2}$-cap $\theta$ such that
\begin{equation}
|S_{\theta, j}f(x)|\gtrapprox|S_{j}f(x)|,
\end{equation}
which fulfills the second part of the lemma and hence indicates the narrowness.   
\end{proof}

\bigskip

\subsection{The preliminary reduction for the model operator}

Using Lemma \ref{broad-narrow}, we partition $F_j$ as 
\begin{equation}
    F_j=\Big(\bigcup_\al E_{\al, j}\Big)\cup E_{R^{-1/2}, j},
\end{equation}
where $E_{\al, j}$ is a collection of the $\al$-broad points in $F_j$, and $E_{R^{-1/2},j}$ is a collection of the narrow points in $F_j$. By pigeonholing, there exists an $\al\in [R^{-1/2}, 1]$ so that

\begin{equation}
    \big\|\sum_jS_jf\Id_{F_j}\Big\|_p^p\lessapprox\big\|\sum_jS_jf\Id_{ E_{\al, j}}\big\|_p^p\lessapprox\big\|\sum_jS_{\tau(x), j}f\Id_{ E_{\al, j}}\big\|_p^p,
\end{equation}
where $\tau(x)$ is an $\alpha$-cap depending on the variable $x$. 
Hence there exists a collection of disjoint sets  $\{F_{\tau,j}\}_{\tau,j}$ with $F_{\tau,j}\subset  E_{\al, j}$ so that
\begin{equation}
    \big\|\sum_jS_{\tau(x), j}f\Id_{E_{\al, j}}\big\|_p^p=\sum_\tau\big\|\sum_jS_{\tau, j}f\Id_{F_{\tau,j}}\big\|_p^p=\sum_\tau\sum_j\big\|S_{\tau, j}f\Id_{F_{\tau,j}}\big\|_p^p,
\end{equation}
where $\tau$ ranges over all $\al$-caps. 
For each $\alpha$-cap $\tau$, we denote the normal vector of the cap $\tau$ at its center by $e(\tau)$.
Divide $B_R$ into parallel tubes of dimensions $R\alpha \times R$, oriented in the direction of $e(\tau)$. 
We denote the collection of these tubes as $\mathbb{T}_\tau$ and let $\ZT:=\bigcup_\tau\ZT_\tau$ where $\tau$ ranges over all $\al$-caps. 
For each $T\in\ZT_\tau$,
\begin{equation}
    T_j:=F_{\tau,j}\cap T.
\end{equation}
It is important to note that the sets $\{T_j : T \in \mathbb{T},\ j \in [1,R] \cap \mathbb{Z}\}$ are pairwise disjoint.
Since $S_{\tau, j}f$ is supported in a unit ball in the frequency space, by the uncertainty principle (see also Lemma \ref{linearization-lem}), $T_j$ can be chosen as a collection of unit balls in the physical space. 
For dyadic numbers $\la\in[1,\al R^2]$ and $\nu\in[1,R]$, let $\ZT_{\la,\nu}\subset\ZT$ denote the collection of tubes $T$ in $\ZT$ satisfying $\# \{j: |T_j|\sim \la\}\sim \nu$. 
We remark that $\{\ZT_{\la,\nu}\}_{\la,\nu}$ need not be disjoint.  
Since for fixed $T\in\ZT_{\la,\nu}$, the sets $\{T_j:j\in[1,R]\cap \ZZ\}$ are disjoint,  it is easy to see that 
\begin{equation}\label{la-nu}
    \la\nu \lesssim \al R^2\,. 
\end{equation}

\medskip


By pigeonholing, there exist $\la,\nu$ and, for each $T\in\ZT_{\la,\nu}$, a set $\cj_T$ such that $|\cj_T|\sim\nu$, $|T_j|\sim\lambda$ for each $T\in\ZT_{\lambda,\nu}$ and $j\in\cj_T$, and
\begin{equation}
\label{after-bn}
    \sum_\tau\sum_j\big\|S_{\tau, j}f\Id_{F_{\tau,j}}\big\|_p^p\lessapprox\sum_\tau\sum_{T\in\ZT_\tau\cap\ZT_{\la,\nu}}\sum_{j\in\cj_T}\big\|S_{\tau, j}f\Id_{T_j}\big\|_p^p.
\end{equation}
We use $\ZT_{\tau, \lambda, \nu}$ to denote $\ZT_\tau\cap \ZT_{\la,\nu}$.  
This finishes our broad-narrow reduction. 
Therefore,  Theorem \ref{main-thm} boils down to the estimate
\begin{equation}
\label{reduction-1}
    \sum_\tau\sum_{T\in\ZT_{\tau, \lambda, \nu}}\sum_{j\in\cj_T}\big\|S_{\tau, j}f\Id_{T_j}\big\|_p^p\lessapprox \big\|f\big\|_p^p
\end{equation}
for $p=5/3$. 

\medskip

Note that $\ZT_{\lambda, \nu}= \bigcup_\tau \mathbb T_{\tau, \lambda, \nu}$.
We observe that the collection $\ZT_{\lambda, \nu}$ of $R\al\times R$ tubes obeys certain non-concentration properties. 


\begin{definition}[One-dimensional ball condition]
Let $\rho_1\leq \rho_2$ be two positive numbers. 
We say  $\ZT$, a collection of, $\rho_1\times\rho_2$ tubes obeys the {\bf one-dimensional ball condition with factor $M$}, if any $\rho\times\rho_2$-tube, where $\rho_1\leq\rho\leq\rho_2$, contains $\leq M (\rho/\rho_1)$ tubes in $\ZT$.
\end{definition}


\begin{lemma}
\label{non-concentration-lem}
$\ZT_{\lambda, \nu}$ is a collection of $R\al\times R$ tubes which obeys one-dimensional ball condition with factor $\lesssim R^2\al/(\la\nu)$.    
\end{lemma}
\begin{proof}
Assume $r\in[R\al, R]$. Let $\bar T$ be an arbitrary $r\times R$ tube, and define  $\ZT(\bar T):=\{T\in\ZT: T\subset\bar T\}$. It suffices to prove that $\#\ZT(\bar T)\lesssim \frac{R^2\al}{\la\nu} \frac{r}{R\al}$. 
Notice that 
\begin{equation}
    \bar T\supset \bigcup_{T\in\ZT(\bar T)}\bigcup_{j\in\cj_T}T_j
\end{equation}
since $T_j$'s are disjoint when $j\in \cj_T$.   Thus, we see that 
\begin{equation}
    rR\sim|\bar T|\geq\sum_{T\in\ZT(\bar T)}\sum_{j\in\cj_T}|T_j|.
\end{equation}
Since $|T_j|\sim\la, \#\cj(T)\sim \nu$, we obtain 
\begin{equation}
    rR\gtrsim \#\ZT(\bar T)\la\nu,
\end{equation}
which is exactly what we want.
\end{proof}

\bigskip

\subsection{Refinement of the reduction process for the model operator}

We now 
make a further reduction on the left-hand side of \eqref{reduction-1}. 
By dyadic pigeonholing, for each $\al$-cap $\tau$, there exist a set $\ZT_{\tau, \la, \nu}'\subset\ZT_{\tau, \la, \nu}$ and a union of finitely overlapping unit balls $T_j'$ for each $T\in\ZT_{\tau,\lambda,\nu}'$ and each $j\in\cj_T$ such that
\begin{enumerate}
    \item $T_j'\subset T_j$, and $|T_j'|$ are the same up to a constant multiple for all $T\in\bigcup_\tau \ZT_{\tau,\lambda,\nu}'$ and each $j\in\cj_T$.
    \item $|S_{\tau, j}f(x)\Id_{T_j'}(x)|$ are the same up to a constant multiple for all $x\in T_j'$.
    \item We have
    \begin{equation}
    \label{reduction-2}
        \sum_\tau\sum_{T\in\ZT_{\tau, \la, \nu}'}\sum_{j\in\cj_T}\big\|S_{\tau, j}f\Id_{T_j'}\big\|_p^p\gtrapprox\sum_\tau\sum_{T\in\ZT_{\tau, \la, \nu}}\sum_{j\in\cj_T}\big\|S_{\tau, j}f\Id_{T_j}\big\|_p^p.
    \end{equation}
\end{enumerate}

To simplify notation, we continue to denote $\ZT_{\tau, \la, \nu}'$ by $\ZT_{\tau, \la, \nu}$, and represent $\ti{T_j}$ as $T_j$. 
Therefore, we have $|T_j| \lesssim \lambda$ for all $T \in \ZT_{\la,\nu}=\bigcup_{\tau} \ZT_{\tau, \la, \nu}$ and $j \in \mathcal{J}_T$.

\smallskip

Now, we shift our focus to a particular $T_j$ and $\big\|S_{\tau, j}f\Id_{T_j}\big\|_p^p$. Through a linear transformation, let's assume, without loss of generality, that $\tau$ represents the cap $\{e\in\mathbb{S}^1: \text{dist}(e,e_2)\leq\alpha\}$, where $e_2=(0,1)$ denotes the vertical unit vector. 
Under this assumption,  the Fourier transform of $S_{\tau, j}f$ is essentially contained in the set 
\begin{equation}
    \{(\xi_1,\xi_2):|\xi_2-(t_j-\xi_1^2)^{1/2}|\leq R^{-1}\text{ and }|\xi_1|\lesssim \al\}.
\end{equation}
Define $\cl$ to be the parabolic rescaling (depending on $T$)
\begin{equation}
\label{para-rescaling}
    \cl(x_1,x_2)=((\al\log R) x_1, (\al\log R)^2 x_2).
\end{equation}
Then, for some $C^\infty$ function $\phi$ with $|\phi''|\sim1$, the Fourier transform of $S_{\tau, j}f\circ\cl$ is contained in the set
\begin{equation}
\label{before-K}
    \{(\xi_1,\xi_2):|\xi_2-\phi(\xi_1)|\lesssim (\al\log R)^{-2}R^{-1}\text{ and }|\xi_1|\lesssim1\}.
\end{equation}

Fix $K$ with $K\sim \exp{(\log R/\log\log R)}\lessapprox1$ such that $K^n=R(\al\log R)^2$ for some $n\leq \log\log R$. In particular, $(\log R)^{O(n)}\lessapprox1$. 
Let us define $E_1:=\cl(T_j)$. 
Then $\cl(T_j)$ is contained in an $R(\al\log R)^2$-ball.
Next, we are going to find a significant subset of $E_1$ whose $K$-neighborhood is regular for some factors.
 
Since $|S_{\tau, j}f(x)\Id_{T_j}(x)|$ are the same up to a constant multiple, $|(S_{\tau, j}f\Id_{T_j})\circ\cl(x)|$ are the same up to a constant multiple for all $x\in E_1$. By dyadic pigeonholing, there exists a set $E_2\subset E_1$ such that
\begin{enumerate}
    \item $|E_2|\gtrapprox|E_1|$
    \item $|B\cap E_2|$ are the same up to a constant multiple for all $K$-balls $B\subset \cl(T)$, if $B\cap E_2\not=\varnothing$.
\end{enumerate}
Let $E_3:=N_K(E_2)$. 
By Lemma \ref{regular-lem-multi}, there exists a union of $K$-balls $E_4\subset E_3$ with $|E_4|\gtrapprox|E_3|$ such that $E_4$ is regular with some factors $(\ka_{r_1},\ldots,\ka_{r_n})$, where $r_1=K, r_n=R\al^2$. 
Denote by $E_2'=E_2\cap E_4$. 
Since $|B\cap E_2|$ are the same up to a constant multiple for all $K$-balls $B\subset \cl(T)$ if $B\cap E_2\not=\varnothing$ and since $|E_4|\gtrapprox|E_3|$, we have $|E_2'|\gtrapprox|E_2|$ and hence $|E_2'|\gtrapprox|E_1|$.

Let $T_j':=\cl^{-1}(E_2')$. Since $|E_2'|\gtrapprox|E_1|$, 
\begin{equation}
    \big\|S_{\tau, j}f\Id_{T_j'}\big\|_p^p\gtrapprox\big\|S_{\tau, j}f\Id_{T_j}\big\|_p^p.
\end{equation}
To ease notation, let us still use $T_j$ to denote $T_j'$. 

\medskip

Refer back to equations \eqref{reduction-1} and \eqref{reduction-2}. 
Therefore, after another dyadic pigeonholing on $T_j$ about the factors $(\ka_{r_1},\ldots,\ka_{r_n})$, demonstrating Theorem \ref{main-thm} boils down to showing 
\begin{proposition}
\label{reduction-lem}
Let $p=5/3$ and let $\alpha\in [R^{-1/2}, 1]$, $\la\in [1, \alpha R^2]$ and $ \nu\in [1, R]$ be dyadic numbers obeying (\ref{la-nu}). 
Suppose $\ZT_{\la, \nu}=\bigcup_{\tau}\ZT_{\tau, \la, \nu}$, where $\tau$ ranges over all $\alpha$-caps, is a collection of $R\al\times R$ tubes such that 
\begin{enumerate}
   \item  for each $T\in\ZT_{\la, \nu}$, there is a set $\cj_T$ with $|\cj_T|\sim\nu$;
    \item for each $T\in \ZT_{\la, \nu}$ and each $j\in\cj_T$,  there is a subset $T_j$ of $T$ such that  $T_j$ is a union of finitely overlapping unit balls and $|T_j|\lesssim\la$.
    Moreover, for all $T \in\ZT_{\la, \nu}$ and $j\in \cj_T$, $T_j$ are pairwise disjoint and  $|T_j|$ are the same up to a constant multiple;
   \item if $\alpha>R^{-1/2}$, then any $x\in T_j$ is $\al$-broad whenever $T\in \ZT_{\la, \nu}$ and $j\in\cj_T$;
    \item $|S_{\tau, j}f(x)\Id_{T_j}(x)|$ are the same up to a constant multiple for all $x\in T_j$;
    \item $N_K(\cl(T_j))$ is regular with some uniform factors $(\ka_{r_1},\ldots,\ka_{r_n})$ for each $T\in\ZT_{\la,\nu}$ and $j\in\cj_T$, where $r_1=K, r_n=R\al^2$.
    \item $\ZT_{\lambda,\nu}$ satisfies Lemma \ref{non-concentration-lem}.
\end{enumerate}
Then, for any $\e>0$, there exists $C_\e>0$ such that
\begin{equation}
\label{reduction-3}
    \sum_\tau\sum_{T\in\ZT_{\tau, \la, \nu}}\sum_{j\in\cj_T}\big\|S_{\tau, j}f\Id_{T_j}\big\|_p^p\leq C_\e R^\e \big\|f\big\|_p^p.
\end{equation}
\end{proposition}


\section{Technical estimates and wrapping up the proof }
\label{three propositions}

In this section, we prove Proposition {\ref{reduction-lem}}.
Fix $\e>0$, and we aim to prove that \eqref{reduction-3} is true for this $\e$.
We first derive certain $L^2$-estimates and $L^{4/3}$-estimates, and then apply interpolation to complete the proof.

\smallskip

Before stating the first $L^2$ estimate, we recall that 
the parabolic rescaling mapping $\mathcal L$ defined in (\ref{para-rescaling}) transforms the image of a $R\al\times R$ tube $T\in\ZT$, denoted as $\cl(T)$,
into a rectangle of dimensions $ R\alpha^2\log R \times R(\alpha\log R)^2$, which is contained in 
an $R(\al\log R)^2$-ball.

\begin{definition}
Let $r\leq R(\al\log R)$. For a given $R\al\times R$ tube $T$, we call an $r\times r(\al\log R)^{-1}$ tube $Q$ an {\bf $r$-pseudo ball}, if $\cl(Q)\subset \cl(T)$ is an $r(\al\log R)$-ball. 
\end{definition}

Initially, we establish a basic yet useful lemma regarding $r$-pseudo balls.

\begin{lemma}\label{lem-5.2}
Let $T\in \mathbb T_{\la, \nu}$ and $\wt T$ be an $R^{1/2}\times R$ tube contained in $T$.
Then 
\begin{equation}
\label{pseudo-ball-count}
    \#\{Q:Q\cap \wt T\cap T_j\not=\varnothing\}\lesssim\ka_n (\al\log R)R^{1/2},
\end{equation}
where $Q$ denotes an $R^{1/2}$-pseudo ball and $T_j$ is given as in Proposition \ref{reduction-lem}.
\end{lemma}

\begin{proof}
Observe that $\cl(\wt T)$ becomes an $R^{1/2}(\al\log R)\times R(\al\log R)^2$ tube. 
 Since $N_K(\cl(T_j))$ is regular with factors $(\ka_{r_1},\ldots,\ka_{r_n})$, we have
\begin{equation}\label{Q5.2}
    \#\{\cl(Q):\cl(Q)\cap \cl(\wt T)\cap \cl(T_j)\not=\varnothing\}\lesssim\ka_n (\al\log R)R^{1/2},
\end{equation}
where $\cl(Q)$, an $R^{1/2}(\al\log R)$-ball, is the image set of an $R^{1/2}$-pseudo ball $Q$. 
Clearly, (\ref{pseudo-ball-count}) follows from (\ref{Q5.2}). 
\end{proof}

\bigskip

\subsection{An application of the refined $L^2$ estimates for the regular set}

The subsequent lemma emerges as a corollary of Proposition \ref{l2-refine-prop}, presenting refined $L^2$-estimates for the regular set.

\begin{lemma}
With the same assumption as Proposition \ref{reduction-lem}, we have
\begin{equation}
\label{l2-1}
    \big\|S_{\tau, j}f\Id_{T_j}\big\|_2^2\lessapprox\prod_{k=1}^n\ka_{r_k}^{-1}\Big(\frac{\la}{R^2\al}\Big)\big\|S_{\tau^*, j}f\big\|_2^2,
\end{equation}
and
\begin{equation}
\label{l2-1-1}
    \big\|S_{\tau, j}(f\Id_{T_j})\big\|_2^2\lessapprox\prod_{k=1}^{n}\ka_{r_k}\big\|f\Id_{T_j}\big\|_2^2.
\end{equation}
In addition, if $Q\subset T$ is an $R^{1/2}$-pseudo ball, then
\begin{equation}
\label{l2-2}
    \big\|S_{\tau, j}(f\Id_{T_j}\Id_Q)\big\|_{2}^2\lessapprox (R^{-1/2}\al^{-1})\prod_{k=1}^{n-1}\ka_{r_k}\|f\Id_{T_j}\Id_Q\|_2^2.
\end{equation}
\end{lemma}
\begin{proof}
When $\alpha = R^{-1/2}$, the desired estimates follow directly from the essentially constant property established in Lemma \ref{essentially-constant}.
For $\alpha> R^{-1/2}$, define $E_5:=\cl (T_j)$.
Note that, unlike $T_j$, $E_5$ may not be a union of unit balls. 
To address this issue, by dyadic pigeonholing, there exists a set $E_6\subset E_5$ such that
\begin{enumerate}
    \item $|E_6|\gtrapprox|E_5|=|\cl (T_j)|$.
    \item $|B\cap E_6|$ are the same up to a constant multiple for all dyadic unit balls $B\subset \cl(T)$, if $B\cap E_6\not=\varnothing$.
\end{enumerate}
Let $T_j':=\cl^{-1}(E_6)$. 
Since $|S_{\tau, j}f(x)\Id_{T_j}(x)|$ are about the same for all $x\in T_j$, 
\begin{equation}\label{5-0}
    \big\|S_{\tau, j}f\Id_{T_j}\big\|_2^2\lessapprox\big\|S_{\tau, j}f\Id_{T_j'}\big\|_2^2,
\end{equation}
and that for any dyadic unit ball $B$ with $B\cap E_6\not=\varnothing$,
\begin{equation}
    \frac{|B\cap E_6|}{|B|}\sim\frac{|E_6|}{|N_1(E_6)|}.
\end{equation}

Let $B$ be a dyadic unit ball with $B\cap E_6\not=\varnothing$. 
Since any $x\in T_j'$ is $\alpha$-broad, there exists a set of finitely overlapping $\al$-caps $\{\si\}$ with $\#\{\si\}\geq 10\log\log R$ such that
\begin{equation}
    |S_{\si, j}f\Id_{T_j'}\circ\cl |\gtrapprox|S_{\tau, j}f\Id_{T_j'}\circ\cl|
\end{equation}
in $E_6\cap B$ for any $\si$. 
Since $S_{\tau, j}f\Id_{T_j'}\circ\cl$, $S_{\si, j}f\Id_{T_j'}\circ\cl$ are essentially constant on a unit ball, we have, for any $\si$,
\begin{equation}\label{5'}
    \big\|S_{\si, j}f\circ\cl\big\|_{L^2(B)}^2\gtrapprox\big\|S_{\tau, j}f\circ\cl\big\|_{L^2(B)}^2\sim \frac{|N_1(E_6)|}{|E_6|}\big\|S_{\tau, j}f\circ\cl\big\|_{L^2(E_6\cap B)}^2.
\end{equation}
 Thus, by taking $M=10\log\log R$ in the broad norm as given in Definition \ref{def-broad-norm}, 
\begin{equation}\label{5''}
    \big\|S_{\tau, j}f\circ\cl\big\|_{L^2(B)}^2\lessapprox\big\|S_{\tau^*, j}f\circ\cl\big\|_{L^2_{\Br_M}(B)}^2.
\end{equation}
Applying (\ref{5'}) and (\ref{5''}) and  summing up all dyadic  unit balls $B$ with $B\cap E_6\not=\varnothing$, 
\begin{equation}\label{5'''}
    \big\|S_{\tau, j}f\Id_{T_j'}\circ\cl\big\|_2^2\lessapprox\frac{|E_6|}{|N_1(E_6)|}\sum_{B,B\cap E_6\not=\varnothing}\big\|S_{\tau^*, j}f\circ\cl\big\|_{L^2_{\Br_M}(B)}^2.
\end{equation}

Let $F:=N_K(E_6)$, which is also a union of $K$-balls $\{B_K\}$. 
Then 
\begin{equation}\label{5-11}
    \sum_{B,B\cap E_6\not=\varnothing}\big\|S_{\tau^*, j}f\circ\cl\big\|_{L^2_{\Br_M}(B)}^2\lesssim K^2\sum_{B_K\subset F}\big\|S_{\tau^*, j}f\circ\cl\big\|_{L^2_{\Br_M}(B_K)}^2.
\end{equation}
Apply \eqref{want-1} in Proposition \ref{l2-refine-prop} (see also Remark \ref{refine-l2-remark}) so that
\begin{equation}\label{5.12}
    \sum_{B_K\subset F}\big\|S_{\tau^*, j}f\circ\cl\big\|_{L^2_{\Br_M}(B_K)}^2\lessapprox\prod_{k=1}^n\ka_{r_k}^{-1}\Big(\frac{|F|}{R^2(\al\log R)^4}\Big)\big\|S_{\tau^*, j}f\circ\cl\big\|_2^2.
\end{equation}
Note that $|F|\lesssim K^2|N_1(E_6)|\lessapprox|N_1(E_6)|$. From (\ref{5-0}), (\ref{5'''}), (\ref{5-11}) and (\ref{5.12}), it follows that 
\begin{equation}
    \big\|S_{\tau, j}f\Id_{T_j}\circ\cl\big\|_2^2\lessapprox\big\|S_{\tau, j}f\Id_{T_j'}\circ\cl\big\|_2^2\lessapprox\prod_{k=1}^n\ka_{r_k}^{-1}\Big(\frac{|E_6|}{R^2(\al\log R)^4}\Big)\big\|S_{\tau^*, j}f\circ\cl\big\|_2^2.
\end{equation}
Recall that $|E_5|\geq |E_6|\gtrapprox|E_5|=|\cl (T_j)|=(\al\log R)^3|T_j|$ and $|T_j|\lesssim\la$. 
Plug these back to the above estimate and rescale to obtain \eqref{l2-1}.

\medskip 

The proof for \eqref{l2-1-1} and \eqref{l2-2} essentially follows the same logic. 
Here, we'll solely focus on verifying \eqref{l2-2}. 
Employing duality, it is enough to establish, for an $L^2$ function $g$,
\begin{equation}
    \big\|S_{\tau, j}g\big\|_{L^2(T_j\cap Q)}^2\lessapprox (R^{-1/2}\al^{-1})\prod_{k=1}^{n-1}\ka_{r_k}\|S_{\tau,j}g\|_2^2.
\end{equation}
Since $\mathcal L(T_j)$ is regular with factors $(\kappa_{r_1},\ldots,\kappa_{r_n})$, it follows that $\cl(T_j\cap Q)$ is regular with factors $(\kappa_{r_1},\ldots,\kappa_{r_{n-1}}, R^{-1/2}(\alpha\log R)^{-1})$. 
By applying \eqref{want-2} from Proposition \ref{l2-refine-prop}, we obtain (\ref{l2-2}), as desired. \qedhere

\end{proof}

\begin{proposition}

\label{l2-prop1}
We have the following $L^2$ estimate
\begin{align}\label{5-16}
    \sum_\tau\sum_{T\in\ZT_{\tau, \la, \nu}}\sum_{j\in\cj_T}\big\|S_{\tau, j}f\Id_{T_j}\big\|_2^2
    \lessapprox\prod_{k=1}^n\ka_{r_k}^{-1}\Big(\frac{\la}{R^2\al}\Big)\big\|f\big\|_2^2.
\end{align}
\end{proposition}

\begin{proof}
Since the kernel of $S_{\tau,j}$ decays rapidly outside the $R\alpha \times R$ tube centered at the origin in the direction $e(\tau)$, with $w_T$ being a weight that decreases rapidly outside $T$, we may strengthen \eqref{l2-1} to the superficially stronger estimate
\begin{equation}
    \big\|S_{\tau, j} f \, \mathbf{1}_{T_j}\big\|_2^2 
    \;\lessapprox\; 
    \prod_{k=1}^n \kappa_{r_k}^{-1} \Big(\frac{\lambda}{R^2 \alpha}\Big) 
    \big\| S_{\tau^*, j} f \cdot w_T \big\|_2^2,
\end{equation}
Note that $\{\tau^*\}_\tau$ are $\log R$-overlapping.
Thus, \eqref{5-16} follows by summing the above estimate over all possible $j, T, \tau$ and then applying Plancherel's theorem.
\end{proof}

\bigskip

\subsection{$L^2$ estimates via the local $L^2$ conclusion and Kakeya-type results}
\label{local-L2-for-model-operator}

We need one more quantitative version of the $L^2$ estimates. 
To achieve this, let us make further reductions on the left-hand side of 
\eqref{reduction-3}.  

\smallskip


Partition $B_R$ into $R\al$-balls $B$'s.  By pigeonholing, we can find a collection of $R\al$-balls, denoted by $\cb$,  and a subset $\ZT'_{\tau, \la, \nu}\subset\ZT_{\tau, \la, \nu}$ so that 
\begin{enumerate}
    \item $\|f\Id_B\|_p$, for all $B\in\cb$,   are the same up to a constant multiple;
    \item for all $T\in\ZT'_{\tau, \la, \nu}$,
    \begin{equation}
        \sum_{j\in\cj_T}\big\|S_{\tau, j}\big(\sum_{B\in\cb}f\Id_B\big)\Id_{T_j}\big\|_p^p
    \end{equation}
    are the same up to a constant multiple;
    \item we have
    \begin{equation}
        \sum_\tau\sum_{T\in\ZT_{\tau, \la, \nu}}\sum_{j\in\cj_T}\big\|S_{\tau, j}f\Id_{T_j}\big\|_p^p\lessapprox\sum_\tau\sum_{T\in\ZT'_{\tau, \la, \nu}}\sum_{j\in\cj_T}\big\|S_{\tau, j}\big(\sum_{B\in\cb}f\Id_B\big)\Id_{T_j}\big\|_p^p.
    \end{equation}
\end{enumerate}
To simplify notation, we continue to use $\ZT_{\tau, \lambda, \nu}$ to represent $\ZT'_{\tau, \lambda, \nu}$ and $\ZT_{\lambda, \nu}$ to represent $\bigcup_\tau\ZT'_{\tau, \lambda, \nu}$.

\medskip

Let $\cb_\be\subset\cb$ be the collection of $R\al$-balls such that $2B$ intersects $\sim\be$ many $R\al\times R$-tubes in $\ZT_{\lambda, \nu}$ for each $B\in\cb_\be$. 
Recall that $\e$ is fixed at the beginning of Section \ref{three propositions}.
By several steps of dyadic pigeonholing, we can find a dyadic number $\be$ and a collection of tubes $\ZT_\be\subset\ZT_{\la,\nu}$ so that 
\begin{enumerate}
    \item $|\ZT_\be|\gtrapprox|\ZT_{\la, \nu}|$.
    \item For each $T\in\ZT_\be$ with direction $e(\tau)$, we have
    \begin{equation}
        \sum_{j\in\cj_T}\big\|S_{\tau, j}\big(\sum_{B\in\cb}f\Id_B\big)\Id_{T_j}\big\|_p^p\lessapprox\sum_{j\in\cj_T}\big\|S_{\tau, j}\big(\sum_{B\in\cb_\be}f\Id_B\big)\Id_{T_j}\big\|_p^p.
    \end{equation}
    \item For each $T\in\ZT_\be$, $\#\{B\in\cb_\be: B\subset R^{\e^2}T\}$ are the same up to a constant multiple.
\end{enumerate}
Denote $\ZT_{\tau,\be}=\ZT_\be\cap\ZT_{\tau, \la, \nu}$. As a result, we have
\begin{align}
\label{reduction-before-local-l2}
\sum_{T\in\ZT_{\la,\nu}}\!\sum_{j\in\cj_T}\!\!\big\|S_{\tau, j}\big(\sum_{B\in\cb}f\Id_B\big)\Id_{T_j}\big\|_p^p &\lessapprox \sum_\tau\!\!\sum_{T\in\ZT_{\tau,\be}}\sum_{j\in\cj_T}\!\big\|S_{\tau, j}\big(\sum_{B\in\cb_\be}f\Id_B\big)\Id_{T_j}\big\|_p^p.
\end{align}

\medskip

Our second proposition in this section is the following $L^2$ estimate.

\begin{proposition}
\label{l2-prop2}
For a fixed $\be$, we have 
\begin{equation}
\label{second-L2}
     \sum_\tau  \sum_{T\in\ZT_{\tau,\be}}\sum_{j\in\cj_T}\big\|S_{\tau, j}\big(\sum_{B\in\cb_\be}f\Id_{ B}\big)\Id_{T_j}\big\|_2^2\lessapprox R^{\e^2}\Big(\frac{R^2\al}{\la\nu}\Big)\be^{-1}\big\|f\big\|_2^2.
\end{equation}
\end{proposition}

The proposition is built on Lemma \ref{local-L2},  a local $L^2$-estimate associated with a single ball, and the following Kakeya-type lemma.

\begin{lemma}
\label{kakeya-type-lem}
For fixed $\be$, we have 
\begin{equation}
    (\#\cb_\be)/(\# \ZT_{\la,\nu})\lesssim[(R^2\al)/(\la\nu)]\cdot(\al\be^2)^{-1}.
\end{equation}
\end{lemma}
\begin{proof}
This follows from Kakeya estimates in $\ZR^2$. Indeed, by the definition of $\cb_\be$,
\begin{equation}\label{5.24}
    \be^2|B|(\#\cb_\be)\lesssim\int \Big(\sum_{T\in\ZT_{\la, \nu}}\Id_{T}\Big)^2\lesssim\sum_{T_1\in\ZT_{\la, \nu}}\sum_{T_2\in\ZT_{\la, \nu}}|T_1\cap T_2|.
\end{equation}
By Lemma \ref{non-concentration-lem}, we see that $\ZT_{\la,\nu}$ is a collection of $R\al\times R$-tubes obeying the one-dimensional ball condition with factor $R^2\al/(\la\nu)$.
Thus, by partitioning the $T_2$ into dyadic groups according to their angles with respect to $T_1$, we have
\begin{equation}\label{5.25}
    \sum_{T_2\in\ZT_{\la,\nu}}|T_1\cap T_2|\lessapprox R^2\al/(\la\nu)|T_1|.
\end{equation}
From (\ref{5.24}) and (\ref{5.25}), it follows that  
\begin{equation}
     \be^2|B|(\#\cb_\be)\lessapprox R^2\al/(\la\nu)|T|(\#\ZT_{\la,\nu}),
\end{equation}
which proves the lemma since $|B|\sim\al|T|$.
\end{proof}

\begin{proof}[Proof of Proposition \ref{l2-prop2}]
Note that the kernel of $S_{\tau, j}$ decays rapidly outside an $R\al\times R$ tube, centered at the origin, with the direction $e(\tau)$. 
Thus, either the tail dominates, in which case there is nothing to prove, or, by Cauchy–Schwarz, for a fixed $R^\alpha \times R$ tube $T \in \bigcup_\tau \ZT_{\tau, \beta}$, we have
\begin{align}
     \sum_{j\in\cj_T}\int\Big|S_{\tau, j}\big(\sum_{B\in\cb_\be}f\Id_{B}\big)\Id_{T_j}\big|^2\lesssim &\,\#\{B\in\cb_\be:B\subset R^{\e^2}T\}\\
     &\cdot \sum_{B\subset R^{\e^2}T}\sum_{j\in\cj_T}\int\big|S_{\tau, j}(f\Id_{B})\Id_{T_j}\big|^2.
\end{align}
Recall that $\#\{B\in\cb_\be:B\subset R^{\e^2}T\}$, for all $T\in\mathbb T_\be$,  are the same up to a constant multiple. Hence
\begin{equation}
    \#\{B\in\cb_\be:B\subset R^{\e^2}T\}\lesssim R^{\e^2}\be(\#\cb_\be)/(\#\ZT_\be),
\end{equation}
and we have
\begin{align}\label{5.29}
    &\sum_{\tau}\sum_{T\in\ZT_{\tau,\be}}\sum_{j\in\cj_T}\big\|S_{\tau, j}\big(\sum_{B\in\cb_\be}f\Id_{ B}\big)\Id_{T_j}\big\|_2^2 \\
    \lesssim&\,\sum_\tau\sum_{T\in\ZT_{\tau,\be}}\#\{B\in\cb_\be:B\subset R^{\e^2}T\}\sum_{B\subset R^{\e^2}T}\sum_{j\in\cj_T}\int\big|S_{\tau, j}(f\Id_{B})\Id_{T_j}\big|^2 \\
    \lesssim&\, R^{\e^2}\frac{\be(\#\cb_\be)}{\#\ZT_\be}\sum_{B\in\cb_\be} \sum_{\tau}\sum_{\substack{T\in\ZT_{\tau,\be}\\ R^{\e^2}T\supset B}}\sum_{j\in\cj_T}\int\big|S_{\tau, j}(f\Id_{B})\Id_{T_j}\big|^2. 
    \end{align}
Note that for fixed $\tau$, there are only $O(R^{\e^2})$ many $T\in\ZT_{\tau,\be}$ such that $B\subset R^{\e^2}T$. Invoke Lemma \ref{local-L2} so that
\begin{equation}\label{5.30}
    \sum_{\tau}\sum_{\substack{T\in\ZT_{\tau,\be}\\R^{\e^2}T\supset B}}\sum_{j\in\cj_T}\int\big|S_{\tau, j}(f\Id_{B})\Id_{T_j}\big|^2\lessapprox R^{\e^2}\al\big\|f\Id_B\big\|_2^2.
\end{equation}
Therefore, combining (\ref{5.29}) and (\ref{5.30}), we end up with
\begin{align}   
    \sum_{\tau}\sum_{T\in\ZT_{\tau,\be}}\sum_{j\in\cj_T}\big\|S_{\tau, j}\big(\!\sum_{B\in\cb_\be}\!f\Id_{ B}\big)\Id_{T_j}\big\|_2^2&\lessapprox R^{\e^2}\al\be(\#\cb_\be)/(\#\ZT_\be)\sum_{B\in\cb_\be}\!\big\|f\Id_B\big\|_2^2\\ &\lessapprox R^{\e^2}\al\be(\#\cb_\be)/(\#\ZT_\be)\big\|f\big\|_2^2,
\end{align}
and the proposition follows from Lemma \ref{kakeya-type-lem} since $\#\ZT_\be\gtrapprox\#\ZT_{\la, \nu}$. 
\end{proof}

\bigskip

\subsection{ $L^{4}$-estimates for the dual operator}

The last proposition in this section provides an $L^{4/3}$ estimate, which we establish by examining its dual form in $L^4$.

\begin{proposition}
\label{l4/3-prop}
Fix $\be$. We have the $L^{4/3}$ estimate 
\begin{equation}
\label{l4/3}
    \sum_{\tau}\sum_{T\in\ZT_{\tau,\la, \nu}}\sum_{j\in\cj_{T}}\int\big|S_{\tau, j}\big(\sum_{B\in\cb_\be}f\Id_{B}\big)\Id_{T_j}\big|^{4/3}\lessapprox R^{O(\e^2)}\be^{1/3}\nu^{2/3}\prod_{k=1}^n\kappa_{r_k}^{2/3}\big\|f\big\|_{4/3}^{4/3}.
\end{equation}
\end{proposition}

\begin{proof}

Using the duality between $L^{4/3}$ and $L^4$, it suffices to prove that
\begin{equation}
    \Big|\!\int\sum_{\tau}\!\!\sum_{T\in\ZT_{\tau,\la, \nu}}\!\sum_{j\in\cj_{T}}S_{\tau, j} \big(\!\sum_{B\in\cb_\be}\!f\Id_{B}\big)f_{T_j}\Big|\!\lessapprox R^{O(\e^2)}\be^{1/4}\nu^{1/2}\!\prod_{k=1}^n\kappa_{r_k}^{1/2}\|f\|_{\frac{4}{3}}\|\{f_{T_j}\}\|_{L^4(\ell^4)},
\end{equation}
where $\|\{f_{T_j}\}\|_{L^4(\ell^4)}$ denotes 
\begin{equation}
   \bigg(\sum_{\tau}\sum_{T\in \ZT_{\tau, \la, \nu}}\sum_{j\in \mathcal J_T} \int \big| f_{T_j}\big|^4  \bigg)^{1/4}\,,
\end{equation}
and $f_{T_j}$ is a function supported on $T_j$.
Note that 
\begin{align}
    \int\!\sum_{\tau}\!\!\sum_{T\in\ZT_{\tau,\la, \nu}}\!\sum_{j\in\cj_{T}}\!\!S_{\tau, j} \big(\!\!\sum_{B\in\cb_\be}f\Id_{B}\big)f_{T_j} 
    &=\int\!\! f\sum_{\tau}\!\!\sum_{T\in\ZT_{\tau,\la, \nu}}\!\sum_{j\in\cj_{T}}\!\sum_{B\in\cb_\be}\!\!S_{\tau,j}(f_{T_j})\Id_B.
\end{align}
Hence by H\"older's inequality, it suffices to prove that 
\begin{equation}
\label{before-bn-l4}
    \int \big|\sum_{\tau}\sum_{T\in\ZT_{\tau,\la, \nu}}\sum_{j\in\cj_{T}}\sum_{B\in\cb_\be}S_{\tau,j}(f_{T_j})\Id_B\big|^4\lessapprox R^{O(\e^2)}\be\nu^{2}\prod_{k=1}^n\ka_k^{2}\|\{f_{T_j}\}\|_{L^4(\ell^4)}^4.
\end{equation}
We will investigate the integrand of the above estimate, that is, 
\begin{equation}
\label{before-bn}
    \big|\sum_{\tau}\sum_{T\in\ZT_{\tau,\la, \nu}}\sum_{j\in\cj_{T}}\sum_{B\in\cb_\be}S_{\tau,j}(f_{T_j})\Id_B(x)\big|.
\end{equation}

In parallel with the broad-to-narrow reduction outlined in Lemma \ref{broad-narrow}, for every point $x$ within $B_R$, one of the following statements holds concerning \eqref{before-bn}:

\medskip

\noindent{\bf{A:}} $x$ is {\bf $\al_1$-broad} for some dyadic number $R^{-1/2}<\al_1\leq 1$. 
This means that there exist two caps $\tau_1$ and $\tau_1'$ with $(\log R)\alpha_1 \gtrsim \text{dist}(\tau_1,\tau_1') \gtrsim \alpha_1$, and $|\tau_1|,|\tau_1'|\sim\alpha_1$, such that one of the following conditions holds:
\begin{enumerate}
    \item If $\al_1 \geq \al$, then by restricting the summation over $\tau$ in \eqref{before-bn} to those contained in $\tau_1$ and $\tau_1'$, for $\om \in {\tau_1, \tau_1'}$,
    \begin{align}
        \eqref{before-bn}\lessapprox\big|\sum_{\tau\subset\om}\sum_{T\in\ZT_{\tau,\la, \nu}}&\sum_{j\in\cj_{T}}\sum_{B\in\cb_\be}S_{\tau,j}(f_{T_j})\Id_B(x)\big|.
    \end{align}
    \item If $\al_1<\al$, then there exists an $\al$-cap $\tau$ such that $\tau_1,\tau_1'\subset\tau$, and for $\om\in\{\tau_1,\tau_1'\}$, we have
    \begin{align}
        \eqref{before-bn}\lessapprox\big|\sum_{T\in\ZT_{\tau,\la, \nu}}&\sum_{j\in\cj_{T}}\sum_{B\in\cb_\be}S_{\om,j}(f_{T_j})\Id_B(x)\big|.
    \end{align}
\end{enumerate}

\noindent{\bf{B}:} $x$ is {\bf narrow}. This means that there exists an $R^{-1/2}$-cap $\theta$ and an $\al_1$-cap $\tau_1$ such that $\theta\subset\tau_1$ and
\begin{align}
    \eqref{before-bn}\lessapprox\big|\sum_{T\in\ZT_{\tau,\la, \nu}}\sum_{j\in\cj_{T}}\sum_{B\in\cb_\be}S_{\theta,j}(f_{T_j})\Id_B(x)\big|.
\end{align}
The proof of this dichotomy is almost identical to the one in Lemma \ref{broad-narrow}. More precisely, 
this broad-narrow process can be carried out through an iteration commencing with the initial scale for $\alpha_1$ approximately equal to 1. It concludes when reaching the scale of $R^{-1/2}$.  We omit its details here.  

\medskip 

Hence we see that there is a dyadic number $\alpha_1\in [R^{-1/2},1]$ such that 
\begin{align}
\label{before-case-study}
    \int \big|\eqref{before-bn}\big|^4\lessapprox\int_{E_{\al_1}} \big|\sum_{\tau}\sum_{T\in\ZT_{\tau,\la, \nu}}\sum_{j\in\cj_{T}}\sum_{B\in\cb_\be}S_{\tau,j}(f_{T_j})\Id_B\big|^4,
\end{align}
where  $E_{\alpha_1}$ refers to the set of $\alpha_1$-broad points if $\alpha_1$ falls within the interval $(R^{-1/2}, 1]$, and it signifies the set of narrow points if $\alpha_1=R^{-1/2}$. 

\medskip

Let's examine two distinct scenarios based on the magnitude of $\alpha_1$:

\begin{itemize}
\item [$\bullet$] Case I, where $\alpha_1\geq \alpha$;
\item [$\bullet$] Case II, where $\alpha_1<\alpha$.
\end{itemize}

\medskip

\noindent{\bf{Case I}},  where $\al_1\geq\al$. 
In this case, we begin by applying Lemma \ref{l4-ortho-lem-1} to bound the right-hand side of \eqref{before-case-study} by a square function at scale $\alpha$. 
Next, Lemma \ref{l4-ortho-lem-2} allows us to further control this coarse square function by one at the finer scale $R^{-1/2}$. 
We then invoke the bilinear Kakeya estimate to bound the resulting square function by an $L^2$ norm, in a standard manner. 
Finally, we recover the desired $L^4$ bound from the $L^2$ estimate via Hölder's inequality.

\smallskip

We now turn to the details.
Partition the interval $[1/2,2]$ into $\al\al_1^{-1}$-intervals $\{I\}$. For each $T\in\ZT_{\tau,\la, \nu}$, decompose $\cj_T$ as a union of $\cj_{T}(I)$, where $\cj_{T}(I)=\{j\in\cj_T:t_j\in I\}$. By dyadic pigeonholing, there exists a dyadic number $\mu$, a tube set $\ZT_{\tau,\la, \mu, \nu}$, and an index set $\cj_{T,\mu}$ for each $T\in\ZT_{\tau,\la, \mu, \nu}$ such that
\begin{enumerate}
    \item $\#\cj_{T,\mu}(I)\sim\mu$ whenever $\cj_{T,\mu}(I)=\{j\in\cj_{T, \mu}: t_j\in  I\}\not=\varnothing$.
    \item We have
    \begin{equation}
        \int \big|\eqref{before-bn}\big|^4\lessapprox\int_{E_{\al_1}} \big|\sum_{\tau}\sum_{T\in\ZT_{\tau,\la, \mu, \nu}}\sum_{j\in\cj_{T,\mu}}\sum_{B\in\cb_\be}S_{\tau,j}(f_{T_j})\Id_B\big|^4.
    \end{equation}
\end{enumerate}
Note that for each $T\in \ZT_{\tau, \la, \nu}$,  $\#\cj_T\sim\nu$. Thus we have
\begin{equation}
    \#\{I:\cj_{T, \mu}(I)\not=\varnothing\}\lesssim\nu/\mu.
\end{equation}

We provide proof details exclusively for the scenario where $\alpha_1 > R^{-1/2}$, as the narrow case can be addressed in a similar manner.
By the definition of $\al_1$-broad, we see that 
\begin{align}
   \int \big|\eqref{before-bn}\big|^4 \lesssim\sum_{\tau_1}\sum_{\substack{\tau_1'\\ \dist(\tau_1,\tau_1')\approx \al_1 }}\sum_{B\in\cb_\be}
  \mathfrak S_{\tau_1, \tau_1', B},
\end{align}
where $\tau_1, \tau'_1$ are $\alpha_1$-caps and $\mathfrak S_{\tau_1, \tau_1', B}$ is defined as 
\begin{equation}
\label{before-l4-ortho}
    \int_B\big|\sum_{\tau\subset\tau_1}\sum_{T\in\ZT_{\tau,\la, \mu, \nu}}\sum_{j\in\cj_{T, \mu}}S_{\tau,j}(f_{T_j})\big|^2\big|\sum_{\tau'\subset\tau_1'}\sum_{T'\in\ZT_{\tau',\la, \mu, \nu}}\sum_{j'\in\cj_{T',\mu}}S_{\tau',j'}(f_{T_{j'}'})\big|^2\,.
\end{equation}
Here $\tau, \tau'$ are $\alpha$-caps. 
Since, for a fixed $\tau$, the kernel of $S_{\tau, j}$ decays rapidly outside an $R^\alpha \times R$ tube centered at the origin and parallel to $T_j$, we can move the summations over $T$ and $T'$ outside the squares using the triangle inequality.
Recall that each $B\in\cb_\be$ is associated with $\lesssim \be$ many $\tau$.
Apply Lemma \ref{l4-ortho-lem-1} to \eqref{before-l4-ortho} so that 
\begin{align}
\nonumber
      \eqref{before-l4-ortho} &\lessapprox \be(\nu/\mu)\cdot\\
     &  \sum_{ \substack{\tau\subset\tau_1 \\ \nonumber \tau'\subset\tau_1'}}\sum_{\substack{T\in\ZT_{\tau,\la, \mu, \nu}\\T'\in\ZT_{\tau',\la, \mu, \nu}}}\sum_{I,I'}\int_{2B}\big|\sum_{j\in\cj_{T, \mu}(I)}S_{\tau,j}(f_{T_j})\big|^2\big|\sum_{j'\in\cj_{T', \mu}(I')}S_{\tau',j'}(f_{T_{j'}'})\big|^2.
\end{align}
Summing up all $\tau_1,\tau_1'$ and $B\in\cb_\be$, we obtain that 
\begin{align}
    & \,\,\,\,\, \int \big|\eqref{before-bn}\big|^4\\ \nonumber
    \lesssim & \,\be(\nu/\mu) \!\!\!\!\!\!\!\!\!\!\!\!
    \sum_{\substack{\tau_1, \tau_1'\\ \dist(\tau_1,\tau_1')\approx \al_1}} \!\!
    \sum_{\substack{\tau\subset\tau_1\\ \tau'\subset\tau_1'}}\!
    \sum_{\substack{T\in\ZT_{\tau,\la, \mu, \nu} \\ T'\in\ZT_{\tau',\la, \mu, \nu}}}\!\!
    \sum_{I,I'}
    \int\!\big|\!\!\!\!\sum_{j\in\cj_{T,\mu}(I)}\!\!\!\!\!S_{\tau,j}(f_{T_j})\big|^2\big|\!\!\!\!\sum_{j'\in\cj_{T',\mu}(I')}\!\!\!\!\!\!S_{\tau',j'}(f_{T_j'})\big|^2.
\end{align}

Apply the Cauchy-Schwarz inequality to the index sets $\mathcal{J}_{T, \mu}(I)$ and $\mathcal{J}_{T',\mu}(I')$, then invoke Lemma \ref{l4-ortho-lem-2}. 
Consequently, we arrive at 
\begin{align}
    \nonumber \int \!\!\big|\eqref{before-bn}\big|^4 \!\lesssim &\, \be\nu\mu  \!\!\!\!\!\!\!\!\!\!\!\!
    \sum_{\substack{\tau_1, \tau_1'\\ \dist(\tau_1,\tau_1')\approx \al_1}} \!\!
    \sum_{\substack{\tau\subset\tau_1\\ \tau'\subset\tau_1'}}\!
    \sum_{\substack{T\in\ZT_{\tau,\la, \mu, \nu} \\ T'\in\ZT_{\tau',\la, \mu, \nu}}}   
    \sum_{\substack{j\in\cj_T\\j'\in\cj_{T'}}}
    \sum_{\substack{\theta\subset\tau\\\theta'\subset\tau'}}\!
    \int\!\big|S_{\theta,j}(f_{T_j})\big|^2\big|S_{\theta',j'}(f_{T_{j'}'})\big|^2.
\end{align}
For each $\tau_1$, partition $B_R$ into $R\al_1\times R$-rectangles $\{T_{\tau_1}\}$, pointing the direction $e(\tau_1)$.
Since, for $\theta \subset \tau_1 \cup \tau_1'$, the kernel of $S_{\theta,j}$ decays rapidly outside an $R^{\alpha_1} \times R$ rectangle centered at the origin and oriented along $e(\tau_1)$, at a cost of $O(1)$, we can insert a cutoff in $S_{\theta,j}(f_{T_j})$ as $S_{\theta,j}(f_{T_j} \Id_{T_{\tau_1}})$ and obtain
\begin{align}\label{before-sum-up-1} 
    & \int \big|\eqref{before-bn}\big|^4
    \lesssim \be\nu\mu \cdot\\\nonumber
    \sum_{\tau_1, T_{\tau_1}}&\!\!\!\!
    \sum_{\substack{\tau_1'\\ \dist(\tau_1,\tau_1')\approx \al_1}} \!\!
    \sum_{\substack{\tau\subset\tau_1\\ \tau'\subset\tau_1'}}\!
    \sum_{\substack{T\in\ZT_{\tau,\la, \mu, \nu} \\ T'\in\ZT_{\tau',\la, \mu, \nu}}}\!   
    \sum_{\substack{j\in\cj_T\\j'\in\cj_{T'}}}
    \sum_{\substack{\theta\subset\tau\\\theta'\subset\tau'}}\!\int\!  \big| S_{\theta, j}(f_{T_j}\Id_{T_{\tau_1}})\big|^2 \big|S_{\theta', j'}(f_{T_{j'}'}\Id_{T_{\tau_1}})\big|^2.
\end{align}

Given that $|S_{\theta, j}(f_{T_j}\Id_{T_{\tau_1}})|$ remains nearly constant within any $R^{1/2}\times R$-tube oriented along $e(\theta)$, and considering the bilinear condition $\dist(\tau_1,\tau_1')\gtrapprox \alpha_1$, it becomes evident that 
\begin{equation}\label{**}
    \int |S_{\theta, j}(f_{T_j}\Id_{T_{\tau_1}})|^2|S_{\theta', j'}(f_{T'_{j'}}\Id_{T_{\tau_1}})|^2\lessapprox R^{-2}\al_1^{-1}\big\|S_{\theta, j}(f_{T_j}\Id_{T_{\tau_1}})\big\|_2^2\big\|S_{\theta', j'}(f_{T_{j'}'}\Id_{T_{\tau_1}})\big\|_2^2.
\end{equation}
In the specific scenario where $\al_1=R^{-1/2}$, it's worth mentioning that (\ref{**}) remains valid even without relying on the bilinear structure. 

\medskip

Note that $\#\{\tau_1', \dist(\tau_1,\tau_1')\approx \al_1\}\lessapprox1$. 
We can thus sum up all $\tau\subset\tau_1,\theta,T,j$ and all $\tau'\subset\tau_1',\theta',T',j'$ in \eqref{before-sum-up-1} to get
\begin{align}
  \nonumber  \int \big|\eqref{before-bn}\big|^4\lessapprox \be\nu\mu R^{-2}\al_1^{-1}\sum_{\tau_1}\sum_{T_{\tau_1}}\bigg(\sum_{\tau\subset\tau_1}\sum_{T\in\ZT_{\tau,\la, \mu, \nu}}\sum_{j\in\cj_T}\big\|S_{\tau, j}(f_{T_j}\Id_{T_{\tau_1}})\big\|_2^2\bigg)^2.
\end{align}
Invoke \eqref{l2-1-1} to obtain
\begin{equation}
\label{use-refined-l2}
    \big\|S_{\tau, j}(f_{T_j}\Id_{T_{\tau_1}})\big\|_2^2\lesssim\prod_{k=1}^n\ka_{r_k}\big\|f_{T_j}\Id_{T_{\tau_1}}\big\|_{2}^2,
\end{equation}
which yields
\begin{align}
\nonumber
    \int \big|\eqref{before-bn}\big|^4\lessapprox \be\nu\mu\prod_{k=1}^n\ka_{r_k}^2R^{-2}\al_1^{-1}\sum_{\tau_1}\sum_{T_{\tau_1}}\bigg(\sum_{\tau\subset\tau_1}\sum_{T\in\ZT_{\tau,\la, \mu, \nu}}\sum_{j\in\cj_T}\big\|f_{T_j}\Id_{T_{\tau_1}}\big\|_{2}^2\bigg)^2.
\end{align}
Because the supports of $f_{T_j}$'s are disjoint, employing the Cauchy-Schwarz inequality yields that 
\begin{align}
    \bigg(\sum_{\tau\subset\tau_1}\sum_{T\in\ZT_{\tau, \la, \mu, \nu}}\sum_{j\in\cj_T}\big\|f_{T_j}\Id_{T_{\tau_1}}\big\|_{2}^2\bigg)^2&=\bigg\|\sum_{\tau\subset\tau_1}\sum_{T\in\ZT_{\tau, \la,\mu,\nu}}\sum_{j\in\cj_T}f_{T_j}\Id_{T_{\tau_1}}\bigg\|_{2}^4\\
    &\lesssim R^2\al_1\bigg\|\sum_{\tau\subset\tau_1}\sum_{T\in\ZT_{\tau, \la,\mu,\nu}}\sum_{j\in\cj_T}f_{T_j}\Id_{T_{\tau_1}}\bigg\|_{4}^4.
\end{align}
By summing up all $T_{\tau_1}$ for $\tau_1$, and observing that $\mu \lesssim \nu$, we ultimately obtain
\begin{equation}
    \int \big|\eqref{before-bn}\big|^4\lesssim\be\nu^2\prod_{k=1}^n\ka_{r_k}^2\big\|\{f_{T_j}\}\big\|_{L^4(\ell^4)}^4\,.
\end{equation}
This leads to \eqref{before-bn-l4} and completes the discussion for the case where $\alpha_1\geq\alpha$.

\medskip 

\noindent{\bf{Case II}}, where $\al_1<\al$. 
Recall that $\varepsilon$ is fixed at the beginning of Section \ref{three propositions}. 
We assume that $\alpha_1 > R^{-1/2 + \varepsilon^2}$. 
The remaining narrow case $\alpha_1 \in [R^{-1/2},\, R^{-1/2 + \varepsilon^2})$ is simpler and can be handled in a similar manner, at the cost of a factor of $R^{O(\varepsilon^2)}$ arising from the triangle inequality.

For the case $\alpha_1 > R^{-1/2 + \varepsilon^2}$, we argue similarly as in Case I and obtain \eqref{before-finer-parti}.  
The treatment of \eqref{before-finer-parti} diverges from Case I: if we were to bound its right-hand side as in \eqref{use-refined-l2}, we would face the difficulty of summing all terms $\|f_{T_j}\mathbf{1}_{2T_{\tau_1}}\|_2^2$ over all $T_{\tau_1}$ contained in $T$.  

To overcome this issue, we employ a Carleson-Sj\"olin type argument and further decompose the $\alpha R \times R$ rectangle $T$ into vertical $R\alpha \times R^{1/2}(\alpha \log R)^{-1}$ rectangles $\{Y\}$, where the scale $R^{1/2}(\alpha \log R)^{-1}$ is chosen to match the dimensions of $R^{1/2}$-pseudo balls (see Figure \ref{Fig2}).  
When we restrict $f_{T_j}\mathbf{1}_{2T_{\tau_1}}$ to an $R\alpha \times R^{1/2}(\alpha \log R)^{-1}$ rectangle $Y$, the aforementioned difficulty disappears, allowing us to sum over all corresponding $T_{\tau_1}$ without issue.

\smallskip

Let us turn to the details.
Since $\al_1<\al$, there is only one $\tau$ contributing in \eqref{before-bn}, so we may discard the information from $\cb_\be$. By the definition of $\al_1$-broad, we have
\begin{align}
 \nonumber   \int\! \big|\eqref{before-bn}\big|^4\lesssim&
    \!\sum_{\tau}\!\!\sum_{T\in\ZT_{\tau,\la, \nu}}\!\sum_{\tau_1\subset\tau}\!\!\!\!\!\!\!\sum_{\substack{\tau'_1\\ \dist(\tau_1,\tau_1')\approx \al_1}}
   \! \!\!\!\!\!\!\!\int \!\big|\sum_{j\in\cj_{T}}S_{\tau_1,j}(f_{T_j})\big|^2\big|\sum_{j'\in\cj_{T}}S_{\tau_1',j'}(f_{T_{j'}})\big|^2.
\end{align}
Apply Cauchy-Schwarz's inequality and then Lemma \ref{l4-ortho-lem-2} on the $L^4$ orthogonality to get 
\begin{align}
    \int\! \big|\eqref{before-bn}\big|^4\!\lesssim&
     \!\sum_{\tau}\!\!\!\sum_{T\in\ZT_{\tau,\la, \nu}}\!\sum_{\tau_1\subset\tau}\!\!\!\!\!\!\!\!\!\sum_{\substack{\tau'_1\\ \dist(\tau_1,\tau_1')\approx \al_1}}  
    \sum_{\substack{j\in\cj_{T}\\{j'\in\cj_{T}}}}\!
    \sum_{\substack{\theta\subset\tau_1\\{\theta'\subset\tau_1'}}}
    \!\!\nu^2\!\!\int\!\!\big|S_{\theta,j}(f_{T_j}\!)\big|^2\big|S_{\theta',j'}(f_{T_{j'\!}})\big|^2.
\end{align}
For every $\tau_1$ and each tube $T$ with dimensions $R\alpha \times R$, divide $T$ into $R\alpha_1 \times R$ rectangles $\{T_{\tau_1}\}$, oriented in the direction of $e(\tau_1)$.
Since $\al\geq R^{-1/2+\e^2}$, using the decay property of the kernel $S_{\theta,j}$, analogous to our approach in the first scenario, we obtain that  
\begin{align}
    \int \big|\eqref{before-bn}\big|^4\lesssim&\,\nu^2\sum_{\tau}\sum_{T\in\ZT_{\tau, \la, \nu}}\sum_{\tau_1\subset\tau}\sum_{\substack{\tau'_1\\ \dist(\tau_1,\tau_1')\approx \al_1}}\sum_{j\in\cj_{T}}\sum_{j'\in\cj_{T}}\sum_{T_{\tau_1}\subset T}\\ 
  \nonumber  &\sum_{\theta\subset\tau_1}\sum_{\theta'\subset\tau_1'}\int_{T_{\tau_1}}\big|S_{\theta,j}(f_{T_j}\Id_{2T_{\tau_1}})\big|^2\big|S_{\theta',j'}(f_{T_{j'}}\Id_{2T_{\tau_1}})\big|^2.
\end{align}
Since $|S_{\theta, j}(f_{T_j}\Id_{2T_{\tau_1}})|$ is essentially constant on any $R^{1/2}\times R$-tube with the direction of $e(\theta)$ and since $\dist(\tau_1,\tau_1')\gtrsim \al_1$,
\begin{align}
    \int_{T_{\tau_1}} &|S_{\theta, j}(f_{T_j}\Id_{2T_{\tau_1}})|^2|S_{\theta', j'}(f_{T_{j}}\Id_{2T_{\tau_1}})|^2\\
    &\lessapprox R^{-2}\al_1^{-1}\big\|S_{\theta, j}(f_{T_j}\Id_{2T_{\tau_1}})\big\|_{L^2(T_{\tau_1})}^2\big\|S_{\theta', j'}(f_{T_{j}}\Id_{2T_{\tau_1}})\big\|_{L^2(T_{\tau_1})}^2.
\end{align}
Summing up all $\theta, j, \theta', j'$ and noting $\#\{\tau_1', \dist(\tau_1,\tau_1')\approx \al_1\}\lessapprox1$, we get
\begin{align}
\label{before-finer-parti}
    \int \big|\eqref{before-bn}\big|^4\!\lessapprox \frac{ \nu^2}{R^2\al_1}\!\sum_\tau\!\!\sum_{T\in\ZT_{\tau, \la, \nu}}\! \sum_{ \substack{\tau_1\subset\tau\\{T_{\tau_1}\subset T}}}\!\!\!\bigg(\!\sum_{\theta\subset\tau_1\cap\tau}\sum_{j\in\cj_T}\big\|S_{\theta, j}(f_{T_j}\Id_{2T_{\tau_1}})\big\|_{L^2(T_{\tau_1})}^2\bigg)^2.
\end{align}

Let us fix a $T\in\ZT_{\tau, \la, \nu }$. Partition the $T$ into $R^{1/2}$-pseudo balls $\{Q\}$. Since $T_j$ is a regular set with factors $(\ka_1,\ldots,\ka_n)$, we see that  $\#\{Q: Q\cap \wt T\cap T_j\not=\varnothing\}\lessapprox \ka_n\al R^{1/2}$ for any $R^{1/2}\times R$ tube $\wt T$ by \eqref{pseudo-ball-count} in Lemma \ref{lem-5.2}. Let $\{\wt T\}$ be a collection of $R^{1/2}\times R$ tubes with direction $e(\theta)$ that form a finitely overlapping partition of the $R\al\times R$ tube $T$. Since the kernel of $S_{\theta, j}$ decays rapidly outside an $R^{1/2}\times R$ tube, centered at the origin, with the direction of $e(\theta)$, we have
\begin{equation}\label{.}
    \big\|S_{\theta, j}(f_{T_j}\Id_{2T_{\tau_1}})\big\|_{L^2(T_{\tau_1})}^2\lesssim\sum_{\wt T}\big\|S_{\theta, j}(f_{T_j}\Id_{2T_{\tau_1}}\Id_{\wt T})\big\|_{L^2(T_{\tau_1})}^2.
\end{equation}
From Cauchy-Schwarz's inequality and estimate \eqref{pseudo-ball-count} in Lemma \ref{lem-5.2}, it follows that 
\begin{equation}\label{..}
    \big\|S_{\theta, j}(f_{T_j}\Id_{2T_{\tau_1}}\Id_{\wt T})\big\|_{L^2(T_{\tau_1})}^2\lessapprox\ka_n
    \al R^{1/2}\sum_{Q\subset \wt T}\big\|S_{\theta, j}(f_{T_j}\Id_Q\Id_{2T_{\tau_1}})\big\|_{L^2(T_{\tau_1})}^2.
\end{equation}
Combining (\ref{.}) and (\ref{..}), we end up with 
\begin{align}
\label{kappa-n}
    \big\|S_{\theta, j}(f_{T_j}\Id_{2T_{\tau_1}})\big\|_{L^2(T_{\tau_1})}^2\lessapprox\ka_n
    \al R^{1/2}\sum_{Q\subset T}\big\|S_{\theta, j}(f_{T_j}\Id_{Q}\Id_{2T_{\tau_1}})\big\|_{L^2(T_{\tau_1})}^2.
\end{align}

Let $\{Y\}$ be a cover of the $R\al\times R$-tube $T$ using $R\al\times R^{1/2}(\al\log R)^{-1}$-rectangles, whose longer side is parallel to the shorter side of $T$. 
Then each $Y$ is quantitatively transverse to $T$ and any $T_{\tau_1}$. From \eqref{kappa-n},  we get
\begin{align}
 \nonumber   &\sum_{\tau_1\subset\tau}\sum_{T_{\tau_1}}\bigg(\sum_{\theta\subset\tau_1\cap\tau}\sum_{j\in\cj_T}\big\|S_{\theta, j}(f_{T_j}\Id_{2T_{\tau_1}})\big\|^2_{L^2(T_{\tau_1})}\bigg)^2\\
  \nonumber  \lessapprox&\,\ka_n^2\al^2R\sum_{\tau_1\subset\tau}\sum_{T_{\tau_1}}\bigg(\sum_{\theta\subset\tau_1\cap\tau}\sum_{j\in\cj_T}\sum_{Y\subset T}\sum_{Q\subset Y}\big\|S_{\theta, j}(f_{T_j}\Id_{Q}\Id_{2T_{\tau_1}})\big\|_{L^2(T_{\tau_1})}^2\bigg)^2\\ \label{before-finer-parti-2}
    \lesssim&\,\ka_n^2\al^3R^{3/2}\sum_{Y\subset T}\sum_{\tau_1\subset\tau}\sum_{T_{\tau_1}}\bigg(\sum_{\theta\subset\tau_1\cap\tau}\sum_{j\in\cj_T}\sum_{Q\subset Y}\big\|S_{\theta, j}(f_{T_j}\Id_{Q}\Id_{2T_{\tau_1}})\big\|_{L^2(T_{\tau_1})}^2\bigg)^2
\end{align}
via Cauchy-Schwarz's inequality. Let us fix an $R\al\times R^{1/2}(\al\log R)^{-1}$-rectangle $Y$ now. Partition $Y$ further into $R\al_1\times R^{1/2}(\al\log R)^{-1}$-rectangles $Y'$ (see Figure \ref{Fig2}) so that
\begin{equation}\label{48}
    \sum_{Q\subset Y}\big\|S_{\theta, j}(f_{T_j}\Id_{Q}\Id_{2T_{\tau_1}})\big\|_{L^2(T_{\tau_1})}^2=\!\!\!\sum_{\substack{Y'\subset Y,\\Y'\cap T_{\tau_1}\not=\varnothing}}\sum_{Q\subset Y'}\big\|S_{\theta, j}(f_{T_j}\Id_{Q}\Id_{2T_{\tau_1}})\big\|_{L^2(T_{\tau_1})}^2.
\end{equation}
Since $Y$ is quantitatively transverse to $T_{\tau_1}$, there are $O(1)$ $Y'\subset Y$ such that $Y'\cap T_{\tau_1}\not=\varnothing$. From this observation and (\ref{48}), we see that 
\begin{align}
  \nonumber  &\sum_{\tau_1\subset\tau}\sum_{T_{\tau_1}}\bigg(\sum_{\theta\subset\tau_1\cap\tau}\sum_{j\in\cj_T}\sum_{Q\subset Y}\big\|S_{\theta, j}(f_{T_j}\Id_{Q}\Id_{2T_{\tau_1}})\big\|_{L^2(T_{\tau_1})}^2\bigg)^2\\
  \nonumber  \lesssim &\sum_{\tau_1\subset\tau}\sum_{T_{\tau_1}}\sum_{\substack{Y'\subset Y,\\Y'\cap T_{\tau_1}\not=\varnothing}}\bigg(\sum_{\theta\subset\tau_1\cap\tau}\sum_{j\in\cj_T}\sum_{Q\subset Y'}\big\|S_{\theta, j}(f_{T_j}\Id_{Q}\Id_{2T_{\tau_1}})\big\|_{L^2(T_{\tau_1})}^2\bigg)^2\\ \label{double-sum}
    \lesssim &\sum_{\tau_1\subset\tau}\sum_{T_{\tau_1}}\sum_{\substack{Y'\subset Y,\\Y'\cap T_{\tau_1}\not=\varnothing}}\bigg(\sum_{\theta\subset\tau_1\cap\tau}\sum_{j\in\cj_T}\sum_{Q\subset Y'}\big\|S_{\theta, j}(f_{T_j}\Id_{Q})\big\|_2^2\bigg)^2.
\end{align}
In the last inequality, we use that the kernel of $S_{\theta, j}$ decays rapidly outside the $R\al_1\times R$ tube, centered at the origin,  with the  direction of $e(\tau_1)$.  Again, since there are $O(1)$ many $Y'\subset Y$ such that $Y'\cap T_{\tau_1}\not=\varnothing$, the double sum of \eqref{double-sum} on $(T_{\tau_1}, Y')$ can be bounded by a single sum of $Y'$. 
Thus, we have  
\begin{align}
    \nonumber  \text{R.H.S. of } \eqref{double-sum}\lesssim &\sum_{\tau_1\subset\tau}\sum_{Y'\subset Y}\bigg(\sum_{\theta\subset\tau_1\cap\tau}\sum_{j\in\cj_T}\sum_{Q\subset Y'}\big\|S_{\theta, j}(f_{T_j}\Id_{Q})\big\|_2^2\bigg)^2\\
    \nonumber  \lesssim&\sum_{Y'\subset Y}\bigg(\sum_{\tau_1\subset\tau}\sum_{\theta\subset\tau_1\cap\tau}\sum_{Q\subset Y'}\sum_{j\in\cj_T}\big\|S_{\theta, j}(f_{T_j}\Id_{Q})\big\|_2^2\bigg)^2\\ \label{after-bil-1}
    \lesssim&\sum_{Y'\subset Y}\bigg(\sum_{Q\subset Y'}\sum_{j\in\cj_T}\big\|S_{\tau, j}(f_{T_j}\Id_{Q})\big\|_2^2\bigg)^2.
\end{align}

\begin{figure}
\begin{tikzpicture}
  
\draw[dashed] (0,1) rectangle (12,7);

\draw[rotate around={10:(1,1)}] (.7,2.8) rectangle (12.3,4.8);

\draw[dotted,fill=blue, opacity=.1] (0,1) rectangle (3,7);

\draw[dotted,fill=blue, opacity=.15] (0,3) rectangle (3,5);

\draw[rotate around={5:(1,1)}] (.5,3.9) rectangle (11.7,4.4);

\draw[dotted,fill=blue, opacity=.25] (0,4) rectangle (3,4.5);

\node at (10,2) {$T_j$};

\node at (10,6) {$T_{\tau_1}$};

\node at (10,4.94) {$\wt T$};

\node at (2,2) {$Y$};

\node at (2,3.4) {$Y'$};

\node at (2,4.2) {$Q$};

\end{tikzpicture}
\caption{Possible locations of tubes and pseudo-balls}
\label{Fig2}
\end{figure}

For each $\big\|S_{\tau, j}(f_{T_j}\Id_{Q})\big\|_2^2$, \eqref{l2-2} gives
\begin{equation}
    \big\|S_{\tau, j}(f_{T_j}\Id_{Q})\big\|_2^2\lessapprox \al^{-1} R^{-1/2}\prod_{k=1}^{n-1}\ka_{r_k}\big\|f_{T_j}\Id_{Q}\big\|_2^2.
\end{equation}
Plug this back to \eqref{after-bil-1} so that
\begin{align}
    \text{R.H.S. of }\eqref{double-sum}&\lessapprox \al ^{-2}R^{-1}\prod_{k=1}^{n-1}\ka_{r_k}^2\sum_{Y'\subset Y}\bigg(\sum_{Q\subset Y'}\big\|\sum_{j\in \cj_T}f_{T_j}\Id_{Q}\big\|_2^2\bigg)^2\\
    &\lesssim \al^{-2}R^{-1}(R^{3/2}\al_1\al^{-1})\prod_{k=1}^{n-1}\ka_{r_k}^2\sum_{Y'\subset Y}\Big\|\sum_{j\in \cj_T}f_{T_j}\Id_{Y'}\Big\|_4^4\\
    &\lesssim(\al^{-3}R^{1/2}\al_1)\prod_{k=1}^{n-1}\ka_{r_k}^2\sum_{j\in\cj_T}\big\|f_{T_j}\Id_{Y}\big\|_4^4
\end{align}
by Cauchy-Schwarz's inequality. 
Putting this back into \eqref{before-finer-parti-2}, and subsequently into \eqref{before-finer-parti}, we arrive at
\begin{equation*}
    \int \big|\eqref{before-bn}\big|^4\lessapprox\nu^2\prod_{k=1}^{n}\ka_{r_k}^2\sum_\tau\sum_{T\in\ZT_{\tau, \la, \nu}}\sum_{Y\subset T}\sum_{j\in\cj_T}\big\|f_{T_j}\Id_{Y}\big\|_4^4\lesssim\nu^2\prod_{k=1}^{n}\ka_{r_k}^2\big\|\{f_{T_j}\}\big\|_{L^4(\ell^4)}^4.
\end{equation*}
Since $\be\geq1$, this gives \eqref{before-bn-l4} when $\al_1<\al$.  
Thus, we conclude the proposition.
\qedhere

\end{proof}

\bigskip

\subsection{Completing the proof of Proposition \ref{reduction-lem}}
\label{wrap-up}

Let's establish Proposition \ref{reduction-lem} by utilizing Propositions \ref{l2-prop1}, \ref{l2-prop2}, and \ref{l4/3-prop}.
 Define
\begin{equation}
    I_p(f)^p:=\sum_\tau\sum_{T\in\ZT_{\tau, \la, \nu}}\sum_{j\in\cj_T}\big\|S_{\tau, j}f\Id_{T_j}\big\|_{p}^{p}.
\end{equation}

First, Proposition \ref{l2-prop1} yields 
\begin{align}
\label{final-1}
    I_2(f)^2\lessapprox\prod_{k=1}^n\ka_{r_k}^{-1}\Big(\frac{\la}{R^2\al}\Big)\big\|f\big\|_2^2.
\end{align}

Subsequently, utilizing \eqref{reduction-before-local-l2}, Proposition \ref{l2-prop2}, and Propositions \ref{l4/3-prop}, we obtain
\begin{align}
\label{final-2}
    I_2(f)^2
    \lessapprox R^{\e^2}\Big(\frac{R^2\al}{\la\nu}\Big)\be^{-1}\big\|f\big\|_2^2
\end{align}
and (since $\ZT_{\tau,\be}\subset\ZT_{\tau, \la, \nu}$)
\begin{align}
\label{final-3}
    I_{4/3}(f)^{4/3}\lessapprox \be^{1/3}\nu^{2/3}\prod_{k=1}^n\ka_{r_k}^{2/3}\big\|f\big\|_{4/3}^{4/3}
\end{align}
respectively. 

\smallskip

Given that the Fourier multiplier of $S_{\tau,j}$ is confined within a unit ball, we may assume that $\widehat{f}$ is supported in $B^2(0,2)$, a ball centered at the origin with radius $2$, implying that $|f|$ is essentially constant within any unit ball. Furthermore, as the kernel of $S_{\tau,j}$ exhibits rapid decay beyond the $R$-ball centered at the origin, we can infer that $f$ is also supported in an $R$-ball. Additionally, considering the homogeneity of \eqref{reduction-3} in Proposition \ref{reduction-lem}, we can assume $|f|\leq 1$ and $\sup |f|\sim 1$. Consequently, $\|f\|_{5/3}\gtrsim 1$, as $|f|$ essentially remains constant within any unit ball, notably within the unit ball containing $x$ where $|f(x)|\sim 1$. We partition $f$ as
\begin{equation}
    f=\sum_{k\geq 0}f_k
\end{equation}
where $|f_k|\sim 2^{-k}$ behaves essentially like a characteristic function. 

\medskip

Let's tackle each $f_k$. 
Multiplying the square of \eqref{final-1} with \eqref{final-2} and the cubic of \eqref{final-3}, we obtain
\begin{equation}
\label{f_k}
    I_2(f_k)^6I_{4/3}(f_k)^4\lessapprox  R^{\e^2}\Big(\frac{\la\nu}{R^2\al}\Big)\big\|f_k\big\|_2^6\big\|f_k\big\|_{4/3}^4\lesssim R^{\e^2}\big\|f_k\big\|_{5/3}^{10}\,,
\end{equation}
since $\la\nu\lesssim R^2\al$ by (\ref{la-nu}). 
Employing Hölder's inequality and \eqref{f_k}, we derive 
\begin{equation}\label{!}
    I_{5/3}(f_k)\leq I_2(f_k)^{3/5}I_{4/3}(f_k)^{2/5}\lessapprox  R^{\e^2}\|f_k\|_{5/3}.
\end{equation}

If $2^k\geq R^{10}$, as $f$ is supported in an $R$-ball, we have $\|f_k\|_{5/3}\lesssim 2^{-k/2}\|f\|_{5/3}$. 
Hence, by \eqref{!},  
\begin{align}
   I_{5/3}(f)\leq\sum_{k\geq 0}I_{5/3}(f_k)&\leq \sum_{k=0}^{[10\log R]}I_{5/3}(f_k)+\sum_{k\geq 10\log R}2^{-k/2}\|f\|_{5/3}\\
 &\lesssim(\log R)\|f\|_{5/3}\lessapprox\|f\|_{5/3}.
\end{align}
These establish Proposition \ref{reduction-lem} and consequently Theorem \ref{main}.
\qed

\bigskip

\section{Ending remark: Ruling out Tao's example}
\label{ending-remark}

Tao's example $f(x_1,x_2)=a(R^{-1/2}x_1)a(x_2)e^{ix_1}$ can be  illustrated by Figure \ref{fig3}: On the frequency side, $\wh f$ is a bump function supported on an $R^{-1/2}\times1$ vertical rectangle (depicted as the shaded region in  Figure \ref{fig3}) away from the origin. 
The intersection of this rectangle with any $R^{-1}$-annulus $A_j$ of radius $t_j\sim1$  essentially forms  an $R^{-1/2}\times R^{-1}$ rectangular cap $\theta_j$. 
Specifically, if $\dist(\theta_j,\theta_k)\gtrsim R^{-1/2}$, then $\theta_j,\theta_k$ correspond to different directions. 
For example, $\theta_1,\theta_2,\theta_3$ in Figure \ref{fig3} give distinct directions.

\begin{figure}[hb!t]
\begin{tikzpicture}    
\draw [black, dotted, domain=30:100] plot ({5*cos(\x)}, {5*sin(\x)});
\draw [black, dotted, domain=30:100] plot ({4.8*cos(\x)}, {4.8*sin(\x)});
\draw [black, dotted, domain=30:100] plot ({4*cos(\x)}, {4*sin(\x)});
\draw [black, dotted, domain=30:100] plot ({3.8*cos(\x)}, {3.8*sin(\x)});
\draw [black, dotted, domain=30:100] plot ({3*cos(\x)}, {3*sin(\x)});
\draw [black, dotted, domain=30:100] plot ({2.8*cos(\x)}, {2.8*sin(\x)});

\draw[dotted,fill=blue, opacity=.1] (1.8,1) rectangle (2.5,5);

\draw[color=blue,thick,rotate around={40:(0,0)}] (2.8,-.6) rectangle (3,.6);

\draw[color=blue,thick,rotate around={55:(0,0)}] (3.8,-.5) rectangle (4,.5);

\draw[color=blue,thick,rotate around={63:(0,0)}] (4.8,-.45) rectangle (5,.45);

\draw [->] (0,1.5) -- (0,5.5) node [anchor=south east] {$t$};

\node at (2.3,4.34) {\tiny{$\theta_1$}};
\node at (0.2,4.6) {\small$t_1$};;
\fill (0,4.9) circle (.5mm);

\node at (2.3,3.16) {\tiny{$\theta_2$}};
\node at (0.2,3.6) {\small$t_2$};;
\fill (0,3.9) circle (.5mm);

\node at (2.3,1.77) {\tiny{$\theta_3$}};
\node at (0.2,2.6) {\small$t_3$};;
\fill (0,2.9) circle (.5mm);

\end{tikzpicture}
\caption{}
\label{fig3}
\end{figure}

To exclude Tao's example, one may pose the following question: If the frequency of the function $f$ is well-localized in a horizontal $R^{-1/2}\times 1$-rectangle away from the origin, can we verify Conjecture \ref{mbr-conj}?
We will demonstrate that our approach can handle this case.

\medskip

Let $\vp(x_1,x_2)=a(R^{-1/2}x_1)a(x_2)e^{ix_1}$ be a bump function. Define a family of frequency-localized (with respect to the horizontal rectangle given by $\vp$) operators 
\begin{equation}
    S^\vp_jf:=S_j(\vp\ast f), \hspace{.5cm}1\leq j\leq R.
\end{equation}
\begin{proposition}
Let $\{F_j\}_j$ be a family of disjoint subsets of $B_R$ such that each $F_j$ is a union of finitely overlapping unit balls. 
For $p=3/2$ and any $f\in L^p(\ZR^2)$, we have 
\begin{equation}
\label{main-esti-ending}
    \big\|\sum_{j}S^\vp_jf\Id_{F_j}\big\|_p\lessapprox\big\|f\big\|_p.
\end{equation}
\end{proposition}
\begin{proof}[Sketch of proof]
Notice that the Fourier transform of the kernel of $S^\vp_j$ is supported in an $R^{-1/2}\times R^{-1}$-rectangular cap $\theta_j$. Hence the kernel of $S^\vp_j$ is essentially supported on a $R^{1/2}\times R$-tube with the direction $e(\theta_j)$ centered at the origin. For a fixed $j$, partition $B_R$ into finitely overlapping $R^{1/2}\times R$-tubes with the direction $e(\theta_j)$. Denote this family of $R^{1/2}\times R$-tubes by $\ZT_{\theta_j}$. For each $T\in\ZT_{\theta_j}$, define
\begin{equation}
    T_j:=T\cap F_j.
\end{equation}
Note that the definition of $T_j$ here agrees with the one in Section \ref{Broad-narrow-section}.

\smallskip

By pigeonholing, there exists a subset $\ZT_{\theta_j}'\subset\ZT_{\theta_j}$ for any $j$ and two numbers $\la,\nu$ such that
\begin{enumerate}
    \item $|T_j|\sim \la$ for any $T\in\bigcup_{j}\ZT_{\theta_j}'$ and any $j$.
    \item For a fixed $T\in\bigcup_{j}\ZT_{\theta_j}'$, $\#\cj_T\sim\nu$, where $\cj_T:=\{j:|T_j|\sim\la\}$.
    \item $\big\|\sum_jS^\vp_jf\Id_{F_j}\big\|_p\lessapprox\big\|\sum_j\sum_{T\in\ZT_{\theta_j}'}S^\vp_jf\Id_{T_j}\big\|_p$.
    \item $\la\nu\lesssim R^{3/2}$.
\end{enumerate}
Let $\ZT_{\la,\nu}=\bigcup_j\ZT_{\theta_j}'$. Similar to Section \ref{local-L2-for-model-operator}, there exists a factor $R^{1/2}\gtrsim\be\geq1$, a set of $R^{1/2}$-balls $\cb_\be$, and a tube set $\ZT_\be\subset\ZT_{\la,\nu}$ such that
\begin{enumerate}
    \item $2B$ intersects $\sim\be$ $R^{1/2}\times R$-tubes in $\ZT_{\la,\nu}$ for each $B\in\cb_\be$.
    \item $|\ZT_\be|\gtrapprox|\ZT_{\la, \nu}|$
    \item For each $T\in\ZT_\be$, we have
    \begin{equation}
        \sum_{j\in\cj_T}\big\|S^\vp_{j}\big(\sum_{B\in\cb}f\Id_B\big)\Id_{T_j}\big\|_p^p\lessapprox\sum_{j\in\cj_T}\big\|S^\vp_{j}\big(\sum_{B\in\cb_\be}f\Id_B\big)\Id_{T_j}\big\|_p^p.
    \end{equation}
    \item  For each $T\in\ZT_\be$, $\#\{B\in\cb_\be: B\subset T\}$ are the same up to a constant multiple.
\end{enumerate}
Therefore, Proposition \ref{l2-prop2} implies that
\begin{equation}  
\label{l2-ending-remark}
    \sum_{T\in\ZT_{\be}}\sum_{j\in\cj_T}\big\|S^\vp_{j}\big(\sum_{B\in\cb_\be}f\Id_{ B}\big)\Id_{T_j}\big\|_2^2\lessapprox\Big(\frac{R^{3/2}}{\la\nu}\Big)\be^{-1}\big\|f\big\|_2^2.
\end{equation}

\smallskip

On the other hand, we claim that 
\begin{equation}
\label{l4/3-ending-remark}
    \sum_{T\in\ZT_{\la, \nu}}\sum_{j\in\cj_{T}}\int\big|S^\vp_{j}\big(\sum_{B\in\cb_\be}f\Id_{B}\big)\Id_{T_j}\big|^{4/3}\lessapprox \be^{1/3}\nu^{2/3}\la^{2/3}R^{-1}\big\|f\big\|_{4/3}^{4/3}.
\end{equation}
To prove \eqref{l4/3-ending-remark}, we follow the proof of Proposition \ref{l4/3-prop} until \eqref{use-refined-l2}, with the left-hand side of \eqref{l4/3} being replaced by the left-hand side of \eqref{l4/3-ending-remark}. Since for any function $g$, $S^\vp_jg$ is essentially constant on any $R^{1/2}\times R$-tube in the direction of $e(\theta_j)$, we obtain the following stronger bound in place of \eqref{use-refined-l2}:
\begin{equation}
        \big\|S^\vp_{j}(f_{T_j}\Id_{T_{\tau_1}})\big\|_2^2\lesssim\la R^{-3/2}\big\|f_{T_j}\Id_{T_{\tau_1}}\big\|_{2}^2.
\end{equation}
Then, we continue following the proof of Proposition \ref{l4/3-prop} to conclude \eqref{l4/3-ending-remark}.

\smallskip

Similar to the argument in Section \ref{wrap-up}, \eqref{l2-ending-remark} and \eqref{l4/3-ending-remark} together yield \eqref{main-esti-ending} for $p=3/2$.
\end{proof}

\bibliographystyle{alpha}
\bibliography{bibli}

\end{document}